\documentclass[11pt]{amsproc}
\usepackage{pifont}
\usepackage[mathcal]{euscript}
\usepackage{amsthm}
\usepackage[numbers,sort&compress]{natbib}
\usepackage{hyperref}
\usepackage{cleveref}
\usepackage{amsfonts}
\usepackage{amsmath,amssymb}
\usepackage{indentfirst,latexsym,bm,amsthm,graphicx,colortbl}
\usepackage{fancyhdr}
\usepackage{times}
\usepackage{amsmath,amsfonts,amssymb,amsthm}
\newtheorem{theorem}{Theorem}[section]
\newtheorem{claim}[theorem]{claim}

\newtheorem{conj}[theorem]{Conjecture}
\newtheorem{notation}[theorem]{Notation}
\newtheorem{prop}[theorem]{Proposition}
\newtheorem{theo}[theorem]{Theorem}
\theoremstyle{definition}
\newtheorem{defi}[theorem]{Definition}

\newtheorem{coro}[theorem]{Corollary}
\theoremstyle{remark}
\newtheorem{remark}[theorem]{Remark}
\makeatletter
\newcommand{\rmnum}[1]{\romannumeral #1}
\newcommand{\Rmnum}[1]{\expandafter\@slowromancap\romannumeral #1@}
\makeatother

\numberwithin{equation}{section}

\title{ON THE LENGTH OF GENERATING SETS WITH CONDITIONS ON MINIMAL POLYNOMIAL}
\author{Chengjie Wang}

\begin{document}
\begin{abstract}
Linear upper bounds may be derived by imposing specific structural conditions on a generating set, such as additional constraints on ranks, eigenvalues, or the degree of the minimal polynomial of the generating matrices. This paper establishes a linear upper bound of \(3n-5\) for generating sets that contain a matrix whose minimal polynomial has a degree exceeding \(\frac{n}{2}\), where \(n\) denotes the order of the matrix. Compared to the bound provided in \cite[Theorem 3.1]{r2}, this result reduces the constraints on the Jordan canonical forms. Additionally, it is demonstrated that the bound \(\frac{7n}{2}-4\) holds when the generating set contains a matrix with a minimal polynomial of degree \(t\) satisfying \(2t\le n\le 3t-1\). The primary enhancements consist of quantitative bounds and reduced reliance on Jordan form structural constraints.
\medskip

\raggedright\text{Keywords:} {the full matrix algebra; length of an algebra; generating systems; the degree of minimal polynomial}
\end{abstract}
\maketitle

\section{Introduction}
\label{S1}
In the study of finite-dimensional algebras, invariants such as the length of an algebra play an important role. For the full matrix algebra \( \mathrm{M}_n(\mathbb{F}) \), the problem of computing its length first emerged in the 1950s through the work of Spencer and Rivlin \cite{r10,r11}, who applied the theory of \(3\times3\) matrix algebras to isotropic continuum mechanics. In the general case, the challenge of calculating the length and establishing bounds for it in terms of matrix size was formally posed by Paz in 1984 \cite{r1} and remains unresolved. A central conjecture claims that the length has a linear upper bound of the following form:
\begin{conj}
\label{C1}
Let $\mathbb{F}$ be an arbitrary field. Then \(
    \ell\big(\mathrm{M}_n(\mathbb{F})\big) = 2n - 2.\)
\end{conj}
It has been proven that the aforementioned conjecture holds for matrices of size at most $7$, as shown in \cite{r1,r3,r30, r38}. Moreover, the linear upper bound $2n-2$ can be established for generating sets satisfying specific conditions, as demonstrated in \cite{r2,r3,r4,r5,r7,r8,r12,r21}. For example, Pappacena \cite[Theorem 4.1(b)]{r4} showed in 1997 that the length of a generating set of the full matrix algebra $\mathrm{M}_n(\mathbb{F})$ containing a matrix with $n$ distinct eigenvalues is bounded above by $2n-2$. Constantine and Darnall \cite{r21} (2005) provided a proof of Paz’s conjecture for matrix sets satisfying a modified \textbf{Poincaré-Birkhoff-Witt (PBW)} property. Recent studies have confirmed the upper bound \(2n-2\) for two-element generating sets under specific conditions. Specifically, Longstaff \cite[Theorem 2]{r12} (2004) derived this bound for irreducible pairs $\{A, B\}$, where $B$ is the $n \times n$ strictly upper triangular elementary Jordan matrix. Building upon Longstaff, Niemeyer, and Panaia's foundational framework \cite{r8}, Lambrou and Longstaff \cite{r7} completed the theoretical verification of Paz's conjecture for pairs of complex matrices of size up to $6$. In 2011, the same bound $2n-2$ was established for irreducible pairs containing a rank-one matrix \cite{r5}. Finally, the conjecture is known to hold if the generating set includes a \textbf{non-derogatory} matrix over an arbitrary field $\mathbb{F}$ \cite{r2}. 

The currently known general upper bounds for the length of the matrix algebra $\mathrm{M}_n(\mathbb{F})$ are not linear in $n$. A foundational result by Paz established that $\ell\left( \mathrm{M}_n(\mathbb{F}) \right) \leq \left\lceil \frac{n^2 + 2}{3} \right\rceil$, where $\lceil \cdot \rceil$ denotes the ceiling function. A subquadratic bound was later derived by Pappacena \cite[Corollary 3.2]{r4} in 1997: $\ell\left( \mathrm{M}_n(\mathbb{F}) \right) < n \sqrt{\frac{2n^2}{n-1} + \frac{1}{4}} + \frac{n}{2} - 2$. This bound remained unbroken for over two decades until 2019, when Shitov \cite{r3} significantly improved the upper bound to $\ell(\mathcal{S}) \leq 2n \log_2 n + 4n - 4$, valid for all generating sets $\mathcal{S} \subset \mathrm{M}_n(\mathbb{F})$.

Nevertheless, some linear upper bounds can also be obtained by imposing special conditions on generating sets, as demonstrated in \cite{r2,r4,r6,r13,r29}. In particular, Markova \cite[Theorem 8.2]{r13} proved that if a generating set $\mathcal{S}$ of the full matrix algebra $\mathrm{M}_n(\mathbb{F})$ contains a matrix $A$ with an eigenvalue $\lambda$ whose corresponding Jordan block of maximal size is unique, then $\ell( \mathcal{S})\leq 2n+ \deg A-3$. In 2009, Guterman and Markova \cite[Theorem 6.1]{r29} demonstrated that for a commutative generating set $\mathcal{S}$, the bound $\ell( \mathcal{S})\leq n-1$ holds, extending Paz's results on commutative matrix algebras to arbitrary fields. Subsequent work by Guterman, Laffey, Markova and {\v{S}}migoc \cite{r2} in 2018 derived upper bounds for generating sets of the full matrix algebra $\mathrm{M}_n(\mathbb{F})$ under specific structural constraints. Their key results include:
\begin{itemize}

\item[(1)] For generating sets containing a matrix with the minimal polynomial of degree $n-k$ over an algebraically closed field $\mathbb{F}$, where $2k<n$ and specific Jordan block configurations are satisfied, the authors proved an upper bound of $2n-2+k$. 

\item[(2)] When $n$ is even and the generating sets include a matrix decomposed as $J_{\frac{n}{2}}(\lambda)\oplus J_{\frac{n}{2}}(\lambda)$ (with $J_{k}(\lambda)$ denoting a Jordan block of size $k$), a refined upper bound $\frac{5n}{2}-2$ was established.
\end{itemize}
These results highlight the interplay between matrix invariants (e.g., the degree of the minimal polynomial and Jordan forms) and the efficiency of generating sets in matrix algebras. 

Although Conjecture \ref{C1} remains unresolved, the length can be accurately calculated for many types of matrix sets and appropriate subalgebras. For further details, see \cite{r13,r14,r16,r17,r18,r19,r28,r29} and references therein. The following examples illustrate key developments. In 2005, Markova \cite{r14} proved that the length of the algebra of upper triangular matrices of size $n$ equals $n-1$ and established upper and lower bounds for lengths of direct sums of block triangular matrix algebras. In 2008, Markova \cite{r13} computed lengths for classical commutative matrix subalgebras including Schur algebra and Courter's algebra, while initiating studies on relationships between algebra lengths and their subalgebras. Guterman and Markova \cite{r29} demonstrated in 2009 that maximal commutative matrix subalgebras attain maximal length under inclusion. Markova \cite{r16} further investigated length function properties in 2010, deriving upper bounds for lengths of local algebras. In the same year, Rosenthal \cite{r17} generated the full matrix algebra $\mathrm{M}_n(\mathbb{F})$ by means of an unconventional angle, proving that for every $n$ and field $\mathbb{F}$, there exists $n$ matrices whose length-$2$ words span $\mathrm{M}_n(\mathbb{F})$. Rosenthal posed whether $m$ matrices exist with length-$k$ words spanning the full algebra $\mathrm{M}_n(\mathbb{F})$. In 2014, Guterman, Markova and Sochnev \cite{r19} studied length functions of semimagic matrix algebras for different generating systems. Klep and {\v{S}}penko \cite{r18} resolved the even case of Rosenthal's problem in 2016, leaving the odd case open. In 2022, Kolegov and Markova \cite{r28} investigated lengths of matrix incidence algebras over small finite fields.

The study of determining the lengths of group algebras has been gradually advancing. It has been demonstrated that matrix representations of group algebra are closely linked to resolving the length problem for group algebras and proving Paz's conjecture. In 2019, Guterman and Markova's seminal work \cite{r22} established estimates for the lengths of group algebras over an arbitrary field for all groups of order up to $7$. Specifically, they determined the lengths for the permutation group $S_3$ and the Klein group $V_4$. In the same year, they \cite{r34} also calculated the length of the group algebra of the quaternion group $Q_8$ over an arbitrary field. Guterman, Markova and Khrystik \cite{r32, r31} systematically studied the length of the group algebra over an Abelian group under specific conditions. In 2022, Markova further derived the length of the group algebra of a noncyclic abelian group satisfying certain conditions \cite{r24}. In 2024, the work of Khrystik presented in \cite{r33} determined the length of the group algebra of the direct product of cyclic groups and cyclic $p$-groups over a field of characteristic $p$. In 2025, Shitov further derived an explicit formula for $\ell(\mathbb{F}[G])$ for any finite Abelian group $G$ and field $\mathbb{F}$ \cite{s1}, fully solving a key problem in group algebra theory. For the group algebra over a non-Abelian group, people have carried out relevant research by applying the existing classification work of classical non-Abelian groups. Markova and Khrystik \cite{r35} investigated the length of the group algebra of dihedral groups and solved the problem of length calculation in the semi-simple case. Subsequently, Markova and Khrystik proved in 2021 \cite[Theorem 4.10]{r23} that the length of the group algebra of the dihedral group of order $2^{k+1}$ over an arbitrary field of characteristic $2$ is equal to $2^k$ by using the method in \cite{r32}. In 2025, Shitov \cite{d1} proved that $\ell(\mathbb{F}[D_n])=n$ for any field $\mathbb{F}$ and $n\ge 3$, resolving a long-standing conjecture.

Recent research has focused on the length of non-associative algebras. In 2017, Guterman and Kudryavtsev \cite{r20} introduced the concept of length for non-associative algebras, analogous to the associative case, and computed the length function for the quaternion and octonion algebras as $2$ and $3$, respectively. A sharp upper bound for the length of a non-associative algebra was later established in \cite{r27}. In 2024, Guterman and Zhilina \cite{r26} derived upper bounds for the lengths of descendingly flexible and descendingly alternative algebras. Subsequently, in \cite{r25}, they fully characterized the lengths of standard composition algebras and Okubo algebras.

In this paper, all matrices are defined over an algebraically closed field $\mathbb{F}$. We adopt this assumption to utilize the Jordan canonical form, which is central to our technical arguments. When referencing prior results that hold over arbitrary fields, we explicitly state their field of validity. We contribute the following linear upper bounds:
\begin{itemize}
\item[(1)] For generating sets that contain a matrix with the degree of the minimal polynomial $d>\frac{n}{2}$, we establish the linear upper bound $3n-5$ significantly relaxing the Jordan canonical form constraints required in \cite[Theorem 3.1]{r2}.
\begin{theo}[Theorem \ref{T10}]
Let $\mathbb{F}$ be an algebraically closed field and let $n, t\in\mathbb{Z}^+$, $n=2t$. Let $\mathcal{S}$ be a generating system of $\mathrm{M}_n(\mathbb{F})$. If $m(\mathcal{S})=t+k$, where $k=1,2,\cdots,t$, then $\ell\left(\mathcal{}{S}\right)\le 3n-5$.
\end{theo}
\begin{theo}[Theorem \ref{T11}]
Let $\mathbb{F}$ be an algebraically closed field and let $n,t\in\mathbb{Z}^+$, $n=2t+1$. Let $\mathcal{S}$ be a generating system of $\mathrm{M}_n(\mathbb{F})$. If $m(\mathcal{S})=t+k$, where $k=1,2,\cdots,t+1$, then $\ell\left(\mathcal{S}\right)\le 3n-5$.
\end{theo}
\item[(2)] When the generating set contains a matrix whose the degree of the minimal polynomial $t$ satisfies $2t\le n\le 3t-1$, we prove that its length has an upper bound $\frac{7n}{2}-4$.
\begin{theo}[Theorem \ref{T12}]
Let $\mathbb{F}$ be an algebraically closed field and let $n,t\in\mathbb{Z}^+$. Let $\mathcal{S}$ be a generating system of $\mathrm{M}_n(\mathbb{F})$. If $2t\le n\le 3t-1$ and $m(\mathcal{S})=t$, then $\ell\left(\mathcal{S}\right)\le \frac{7n}{2}-4$.
\end{theo}
\end{itemize}
This extends prior results by leveraging the relationships between the degree of the minimal polynomial and the order of the matrix. The key improvements lie in both the quantitative bounds and the reduced dependence on structural restrictions in the Jordan form analysis.

Our paper is structured as follows. \cref{S1} introduces the background, motivation, and key questions addressed in this work. \cref{S2}, we formalize the foundational concepts, including the definition of the length of an algebra and related notation. While Conjecture \ref{C1} (mentioned earlier) remains unresolved, this paper advances partial progress by identifying conditions under which linear upper bounds can be established for the length of generating systems in the full matrix algebra $\mathrm{M}_n(\mathbb{F})$. Specifically, \cref{S3} delivers the paper's main results: a linear upper bound of $\frac{7n}{2}-4$ for the length of generating systems in $\mathrm{M}_n(\mathbb{F})$, contingent on the condition $m(\mathcal{S})=t$ and $2t\le n\le 3t-1$, where $m(\mathcal{S})$ measures the maximal degree of minimal polynomial in $\mathcal{S}$.

\section{Preliminaries}
\label{S2}
Let $\mathcal{A}$ be a finite-dimensional associative algebra with identity over an arbitrary field $\mathbb{F}$. Since $\mathcal{A}$ is finite-dimensional over $\mathbb{F}$, it is evidently finite generated. Let $\mathcal{S}=\{a_1, a_2, \cdots, a_k\}$ be a finite generating set of this algebra. 

Firstly, it is necessary to define the length of a word of $\mathcal{S}$.

\begin{defi}
Words in $\mathcal{S}$ are products of a finite number of elements from the alphabet $\mathcal{S}$. The length of a word equals the number of factors in the corresponding product. 
\end{defi}

\begin{remark}
  The length of the word $a_{i_1}\cdots a_{i_t}$, where $a_{i_j}\in\mathcal{S}$, is equal to $t$. The unity $1$ of the algebra is considered as a word of length $0$ over $\mathcal{S}$ and also call it the empty word.  
\end{remark}

\begin{defi}
For $i\geq 0$, let $\mathcal{S}^i$ denote the set of all words in the alphabet $\mathcal{S}$ of length less than or equal to $i$. Let $\mathcal{S}^i\backslash\mathcal{S}^{i-1}$ be the set of all words of length $i$ over $\mathcal{S}$, $i\geq 1$.
\end{defi}

\begin{defi}
The linear span of $ \mathcal{S}$ in a vector space over $\mathbb{F}$, denoted by $\langle\mathcal{S}\rangle$, is the set of all finite $\mathbb{F}$-linear combinations.
\end{defi}

The length functions associated with a generating set $\mathcal{S}$ and an algebra $\mathcal{A}$ are formally defined as follows.

\begin{defi}
Let $L_i\left(\mathcal{S}\right)=\langle\mathcal{S}^i\rangle$ denote the linear span of the words in $\mathcal{S}^i$. Note that 
$$L_0\left(\mathcal{S}\right)=\langle 1\rangle = \mathbb{F}$$
for unitary algebras. Also set $L\left(\mathcal{S}\right)$=$\bigcup\limits_{i=0}^{\infty}L_i\left(\mathcal{S}\right)$ be the linear span of all words in the alphabet $\mathcal{S}$.
\end{defi}

\begin{defi}
A set $\mathcal{S}$ is a generating system for $\mathcal{A}$ if and only if $\mathcal{A}$=$L\left(\mathcal{S}\right)$.
\end{defi}

\begin{remark}
    Since $\mathcal{A}$ is finite-dimensional over $\mathbb{F}$ if follows that the chain 
\begin{equation*}
    \mathbb{F}=L_0(\mathcal{S})\subset L_1(\mathcal{S})\subset L_2(\mathcal{S})\subset \cdots
\end{equation*}       
stabilizes. That is, there is an integer $k$ such that $L_k(\mathcal{S})= L_{k+i}(\mathcal{S})$, for all $i\ge 0$. We must have $L_k(\mathcal{S})=\mathcal{A}$ , since $\mathcal{S}$
is a generating set for $\mathcal{A}$.
\end{remark}

\begin{defi}
The length of a generating system $\mathcal{S}$ of a finite-dimensional algebra $\mathcal{A}$ is defined as the minimum non-negative integer $k$ such that $L_k\left(\mathcal{S}\right)=\mathcal{A}$. The length of $\mathcal{S}$ is denoted by 
\begin{equation*}
\ell\left(\mathcal{S}\right)=\mathop{\min}\{k\in\mathbb{Z}^+:L_k\left(\mathcal{S}\right)=\mathcal{A}\}.
\end{equation*}
\end{defi}

\begin{defi}
The length of an algebra $\mathcal{A}$ is defined as the maximum of the lengths of its generating systems $\mathcal{S}$ such that $L\left(\mathcal{S}\right)=\mathcal{A}$. The length of $\mathcal{A}$ is denoted by 
\begin{equation*}
\ell\left(\mathcal{A}\right)=\mathop{\max}\{\ell\left(\mathcal{S}\right):L\left(\mathcal{S}\right)=\mathcal{A}\}.
\end{equation*}
\end{defi}

\begin{notation}
Let $x \in \mathnormal{R}$. The integer part of $x$, denoted $\lfloor{x}\rfloor$, is defined as the largest integer not exceeding $x$. 
\end{notation}

\begin{notation}
Let \(\mathrm{M}_n(\mathbb{F})\) denote the algebra of $\mathnormal{n}\times\mathnormal{n}$ matrices over an arbitrary field $\mathbb{F}$. 

$\bullet$ $E_{ij}:$ The $\left(i, j\right)$-th matrix unit, i.e., the matrix with $1$ at the $\left(i, j\right)$ position and zeros at the remaining positions. 

$\bullet$ $E_n:$ The $n\times n$ identity matrix. 

$\bullet$ $J_k(\lambda):$ A Jordan block of size $k$ with eigenvalue $\lambda$, defined as $J_k(\lambda)=\lambda E_k+\sum_{i=1}^{k-1}E_{i,i+1}$.

$\bullet$ $J_k=J_k(0):$ The nilpotent Jordan block of size $k$.
\end{notation}

\begin{notation}
Let $A\in \mathcal{S}$ and ${\rm deg}\, A$ denote the degree of the minimal polynomial of the element A over the field $\mathbb{F}$. Since the generating system is finite-dimensional, we can set 
\begin{equation*}
    m(\mathcal{S})={\rm max}\{{\rm deg}\, w, w\in \mathcal{S}\}.
    \end{equation*}
\end{notation}


\section{Results}
\label{S3}
\subsection{Effects of Transformations on the Length of Generating Systems}
\ \newline

The following foundational results are adapted from \cite{r6}. We first recall essential properties of algebra lengths to establish the framework for subsequent proofs.

\begin{prop}[\cite{r6}, Proposition 3.1]
Let $\mathcal{A}$ be a finite-dimensional algebra over an arbitrary field $\mathbb{F}$. If $\mathcal{S}$=$\{a_1, a_2, \cdots, a_k\}$ is a generating system of this algebra and $C=\{c_{ij}\}\in\mathrm{M}_k\left(\mathbb{F}\right)$ is an invertible matrix, then the set of the coordinates of the vector
\begin{equation*}
C\left(\begin{array}{c} a_1\\ \vdots \\ a_k\end{array}\right)=
\left(\begin{array}{c} c_{1, 1}a_1+c_{1, 2}a_2+\dots+c_{1, k}a_k \\ \vdots \\ c_{k, 1}a_1+c_{k, 2}a_2+\dots+c_{k, k}a_k\end{array}\right),
\end{equation*}
i.e., the set 
\begin{equation*}
\mathcal{S}_c=\{c_{1, 1}a_1+c_{1, 2}a_2+\dots+c_{1, k}a_k, \dots, c_{k, 1}a_1+c_{k, 2}a_2+\dots+c_{k, k}a_k\},
\end{equation*}
is a system of generators for the algebra $\mathcal{A}$ and $\ell\left(\mathcal{S}_c\right)=\ell\left(\mathcal{S}\right)$.
\end{prop}

\begin{prop}[\cite{r6}, Proposition 3.2]
Let $\mathcal{A}$ be a finite-dimensional unitary algebra over an arbitrary field $\mathbb{F}$. Assume that $\mathcal{S}$=$\{a_1, a_2, \cdots, a_k\}$ is a generating system of this algebra such that $1_\mathcal{A}\notin \langle a_1, a_2, \cdots, a_k\rangle$. Then for any $\gamma_1, \dots, \gamma_k\in \mathbb{F}$, the set $\mathcal{S}_1=\{a_1+\gamma_1 1_\mathcal{A}, \dots, a_k+\gamma_k 1_\mathcal{A}\}$ is a system of generators for the algebra $\mathcal{A}$ and $\ell\left(\mathcal{S}_1\right)=\ell\left(\mathcal{S}\right)$.
\end{prop}

These propositions demonstrate that the length of a generating system is invariant under invertible linear transformations and constant shifts. 

\subsection{Linear Upper Bounds on the Length of Generating Systems under Algebraic Invariants}
\ \newline

If algebraic invariants (such as rank or minimal polynomial degree) are known for a matrix $A\in L_i(\mathcal{S})$, where $i\geq 1$, we can use them to compute and bound the length of $\mathcal{S}$, as formalized in the subsequent theorems.

\begin{claim}[\cite{r3}, Claim 8]
Assume that the minimal polynomial of every matrix in $\mathcal{S}\subset\mathrm{M}_n\left(\mathbb{F}\right)$ has degree at most $2$, then $\ell\left(\mathcal{S}\right)\le 2\log_2 n$.
\end{claim}

\begin{theo}[\cite{r4}, Theorem 4.1(a)]
Let $\mathcal{S}$ be a generating system of $\mathrm{M}_n\left(\mathbb{F}\right)$. If $L_k(\mathcal{S})$ contains a matrix of rank $r>0$ for some $k$, then $\ell\left(\mathcal{S}\right)\le rn+n-r+k-1$.
\end{theo}

\begin{coro}[\cite{r3}, Corollary 7]
Let $\mathcal{S}\subset\mathrm{M}_n\left(\mathbb{F}\right)$ be a generating system and $k\ge2$. If $L_k(\mathcal{S})$ contains a rank-one matrix, then $\ell\left(\mathcal{S}\right)\le 2n+k-4$.
\end{coro}

The Paz's Conjecture under the assumption that the generating system $\mathcal{S}\subset\mathrm{M}_n\left(\mathbb{F}\right)$ contains a nonderogatory matrix is true.

\begin{theo}[\cite{r2}, Theorem 2.4]
Let $\mathbb{F}$ be an arbitrary field, $n\in\mathbb{N}$. If a generating system $\mathcal{S}\subset\mathrm{M}_n\left(\mathbb{F}\right)$ contains a nonderogatory matrix, then $\ell\left(\mathcal{S}\right)\le 2n-2$.
\end{theo}

Furthermore, the bound $2n-2$ also holds in the presence of a matrix with the degree of the minimal polynomial $n-1$.

\begin{theo}[\cite{r2}, Theorem 2.5]
Let $\mathbb{F}$ be an arbitrary field, $n\in\mathbb{N}$. If a generating system $\mathcal{S}\subset\mathrm{M}_n\left(\mathbb{F}\right)$ contains a matrix with minimal polynomial of degree $n-1$, then $\ell\left(\mathcal{S}\right)\le 2n-2$.
\end{theo}

In particular, the following theorem can be used to find length constraints on the generating set $\mathcal{S}$ when the degree of the minimal polynomial exceeds half of the order of the matrix algebra, and when the conditions of certain Jordan blocks are satisfied.

\begin{theo}[\cite{r2}, Theorem 3.1]
Let $\mathbb{F}$ be an algebraically closed field, $n\in \mathbb{N}, k\in\mathbb{Z}^+$, and $2k<n$. Suppose a generating set $\mathcal{S}$ for the matrix algebra $\mathrm{M}_n\left(\mathbb{F}\right)$ contains a matrix $A$ with the degree of minimal polynomial equal to $n-k$. Furthermore, for any eigenvalue $\lambda$ of $A$ the Jordan matrix $A_{\lambda}$ corresponding to $\lambda$ in the Jordan normal form of $A$ satisfies the same relation $2k_{\lambda}<n_{\lambda}$ where $n_{\lambda}$ is the size of $A_{\lambda}$ and $n_{\lambda}-k_{\lambda}$ is the degree of minimal polynomial of $A_{\lambda}$. Then $\ell\left(\mathcal{S}\right)\le 2n-2+k$.
\end{theo}

In Theorem \ref{T10} and Theorem \ref{T11}, we address the problem of reducing constraints imposed on the Jordan canonical form and establish a linear upper bound for the length of generating systems. This result holds under the general condition that $m(\mathcal{S})>\frac{n}{2}$.

\begin{theo}[\cite{r2}, Lemma 3.2]
Let $\mathbb{F}$ be an arbitrary field and $k, n\in\mathbb{N}, n=2k$. Let $\mathcal{S}\subset\mathrm{M}_n\left(\mathbb{F}\right)$ be a generating system. If $\mathcal{S}$ contains the matrix $A=J_k\bigoplus J_k$, then $\ell\left(\mathcal{S}\right)\le \frac{5n}{2}-2$.
\end{theo}

Nevertheless, the current theorem fails to address scenarios where the degrees of minimal polynomials within the generating set do not exceed half the order of the matrix algebra. This specific case will be systematically examined in the subsequent Theorem \ref{T12}.

\subsection{Main results}
\begin{theo}
\label{T10}
Let $\mathbb{F}$ be an algebraically closed field, $n,t\in\mathbb{Z}^+$ with $n=2t$, and let $\mathcal{S}$ be a generating system of $\mathrm{M}_n\left(\mathbb{F}\right)$. If $m(\mathcal{S})=t+k$ for $k\in \{1,2,\cdots,t\}$, then $l\left(\mathcal{S}\right)\le 3n-5$.
\end{theo}

\begin{proof}
Since $m(\mathcal{S})=t+k$, the generating system $\mathcal{S}$ contains a matrix $A$ with the minimal polynomial of degree $t+k$.
\item [(1)]If $A$ has a single eigenvalue $\lambda_1$, then the minimal polynomial of $A$ is $m(\lambda)=(\lambda-\lambda_1)^{t+k}$. By the Jordan Decomposition Theorem, there exists an invertible matrix $P$ such that
\begin{equation*}
\textbf{$P^{-1}AP$}=    
\begin{pmatrix}     
J_{t+k}(\lambda_1) &     \\       
& N_{(t-k)\times (t-k)}     
\end{pmatrix},\
\end{equation*}
where $N-\lambda_1E$ is a nilpotent Jordan matrix of size $t-k$.

The structure of $N$ is 
\begin{equation*}
N=J_{k_1}(\lambda_1)\oplus J_{k_2}(\lambda_1)\oplus \cdots \oplus J_{k_s}(\lambda_1)\oplus \lambda_1E_r,
\end{equation*}
where $t-k=k_1+k_2+\cdots +k_s+r$ and $t-k \ge k_1 \ge k_2 \ge \cdots \ge k_s\ge 2$.

It is evident that there is at most one diagonal block $J_{t+k}(\lambda_1)$ in $P^{-1}AP$, since $t+k> t-k$ for all $k$. Consequently, 
\begin{equation*}
(P^{-1}AP-\lambda_1\mathnormal{E})^{t+k-1}=\mathnormal{E}_{1,t+k},
\end{equation*}
which implies
\begin{equation*}
(A-\lambda_1\mathnormal{E})^{t+k-1}=P\mathnormal{E}_{1,t+k}P^{-1}. 
\end{equation*}
Since the linear span of $\mathnormal{E}, A, \cdots, A^{t+k-1}$ includes the rank-one matrix $P\mathnormal{E}_{1,t+k}P^{-1}$, it follow that 
\begin{equation*}
\ell\left(\mathcal{S}\right)\le 2n+t+k-1-4 \le 2n+2t-5=3n-5,
\end{equation*}
where the first inequality holds by Corollary 3.5.
\item [(2)]If $A$ has two distinct eigenvalues $\lambda_1, \lambda_2$, then the minimal polynomial of $A$ must then be of the form $m(\lambda)=(\lambda-\lambda_1)^{t+k-s}(\lambda-\lambda_2)^s$($\lambda_1$ and $\lambda_2$ appear symmetrically in this situation), where $s\in\mathbb{Z}^+$ and $t+k-s\ge s$.
\item [(\rmnum{1})]If the minimal polynomial of $A$ is $m(\lambda)=(\lambda-\lambda_1)^{t+k-1}(\lambda-\lambda_2)$, by the Jordan Decomposition Theorem, then there exists an invertible matrix $P$ such that 
\begin{equation*}
\textbf{$P^{-1}AP$}=    
\begin{pmatrix}     
J_{t+k-1}(\lambda_1) &     \\         
&\lambda_2\\     
&  & N_{(t-k)\times (t-k)}     
\end{pmatrix}.\
\end{equation*}
There is at most one diagonal block $J_{t+k-1}(\lambda_1)$ in $P^{-1}AP$, since $t+k-1> t-k$ for any $k$. Thus, we have 
\begin{equation*}
(P^{-1}AP-\lambda_1\mathnormal{E})^{t+k-2}(P^{-1}AP-\lambda_2\mathnormal{E})=(\lambda_1-\lambda_2)\mathnormal{E}_{1,t+k-1},
\end{equation*}
which implies
\begin{equation*}
(A-\lambda_1\mathnormal{E})^{t+k-2}(A-\lambda_2\mathnormal{E})=(\lambda_1-\lambda_2)P\mathnormal{E}_{1,t+k-1}P^{-1}. 
\end{equation*}
The linear span of $\mathnormal{E}, A, \cdots, A^{t+k-1}$ contains the rank-one matrix $P\mathnormal{E}_{1,t+k-1}P^{-1}$, so 
\begin{equation*}
\ell\left(\mathcal{S}\right)\le 2n+t+k-1-4 \le 2n+2t-5=3n-5
\end{equation*}
by Corollary 3.5.
\item [(\rmnum{2})]If the minimal polynomial of $A$ is $m(\lambda)=(\lambda-\lambda_1)^{t+k-2}(\lambda-\lambda_2)^2$, by the Jordan Decomposition Theorem, then there exists an invertible matrix $P$ such that 
\begin{equation*}
\textbf{$P^{-1}AP$}=    
\begin{pmatrix}     
J_{t+k-2}(\lambda_1) &     \\         
&J_2(\lambda_2)\\     
&  & N_{(t-k)\times (t-k)}     
\end{pmatrix}.\
\end{equation*} 
The matrix $P^{-1}AP$ contains at most two diagonal blocks $J_{t+k-2}(\lambda_1)$, as $t+k-2\ge t-k$ holds for all $k$. 

\ding{192}If $k>1$, then $t+k-2> t-k$. This implies that $P^{-1}AP$ has only one diagonal block $J_{t+k-2}(\lambda_1)$. Consequently,
\begin{equation*}
(P^{-1}AP-\lambda_1\mathnormal{E})^{t+k-3}(P^{-1}AP-\lambda_2\mathnormal{E})^2=(\lambda_1-\lambda_2)^2\mathnormal{E}_{1,t+k-2},
\end{equation*}
which translates to
\begin{equation*}
(A-\lambda_1\mathnormal{E})^{t+k-3}(A-\lambda_2\mathnormal{E})^2=(\lambda_1-\lambda_2)^2P\mathnormal{E}_{1,t+k-2}P^{-1}. 
\end{equation*}
Since the linear span of $\mathnormal{E}, A, \cdots, A^{t+k-1}$ contains the rank-one matrix $P\mathnormal{E}_{1,t+k-2}P^{-1}$, Corollary 3.5 gives
\begin{equation*}
\ell\left(\mathcal{S}\right)\le 2n+t+k-1-4 \le 2n+2t-5=3n-5.
\end{equation*}

\ding{193}If $k=1$, then $t+k-2= t-k=t-1$. In this case, the Jordan form becomes
\begin{equation*}
\textbf{$P^{-1}AP$}=    
\begin{pmatrix}     
J_{t-1}(\lambda_1) &     \\    
&J_2(\lambda_2)\\
&  & N_{(t-1)\times (t-1)}     
\end{pmatrix}.\
\end{equation*}
Here, $P^{-1}AP$ may contain at most two diagonal blocks $J_{t-1}(\lambda_1)$. 

$\bullet$ Assume that there are two diagonal blocks $J_{t-1}(\lambda_1)$ in $P^{-1}AP$. In this case, 
\begin{equation*}
\textbf{$P^{-1}AP$}=    
\begin{pmatrix}     
J_{t-1}(\lambda_1) &     \\    
&J_2(\lambda_2)\\
&  & J_{t-1}(\lambda_1)     
\end{pmatrix}.\
\end{equation*} 
Then
\begin{equation*}
(P^{-1}AP-\lambda_1\mathnormal{E})^{t-1}(P^{-1}AP-\lambda_2\mathnormal{E})=(\lambda_2-\lambda_1)^{t-1}\mathnormal{E}_{t,t+1},
\end{equation*}
which implies
\begin{equation*}
(A-\lambda_1\mathnormal{E})^{t-1}(A-\lambda_2\mathnormal{E})=(\lambda_2-\lambda_1)^{t-1}P\mathnormal{E}_{t,t+1}P^{-1}. 
\end{equation*}

$\bullet$ Assume that there is at most one diagonal block $J_{t-1}(\lambda_1)$ in $P^{-1}AP$. In this case, $P^{-1}AP$ has an eigenvalue $\lambda_1$ with a single corresponding Jordan block of the maximal size. Consequently,
\begin{equation*}
(P^{-1}AP-\lambda_1\mathnormal{E})^{t-2}(P^{-1}AP-\lambda_2\mathnormal{E})^2=(\lambda_1-\lambda_2)^2\mathnormal{E}_{1,t-1},
\end{equation*}
which implies 
\begin{equation*}
(A-\lambda_1\mathnormal{E})^{t-2}(A-\lambda_2\mathnormal{E})^2=(\lambda_1-\lambda_2)^{2}P\mathnormal{E}_{1,t-1}P^{-1}. 
\end{equation*}

In both cases, the linear span of $\mathnormal{E}, A, \cdots, A^{t}$ contains a rank-one matrix. By Corollary 3.5, we conclude
\begin{equation*}
\ell\left(\mathcal{S}\right)\le 2n+t+k-1-4 \le 2n+2t-5=3n-5.
\end{equation*}

\item [(\rmnum{3})]If the minimal polynomial of $A$ is $m(\lambda)=(\lambda-\lambda_1)^{t+k-s}(\lambda-\lambda_2)^s$, where $3\le s\le \lfloor{\frac{t+k}{2}}\rfloor$, by the Jordan Decomposition Theorem, then there exists an invertible matrix $P$ such that 
\begin{equation*}
\textbf{$P^{-1}AP$}=    
\begin{pmatrix}     
J_{t+k-s}(\lambda_1) &     \\         
&J_s(\lambda_2)\\     
&  & N_{(t-k)\times (t-k)}     
\end{pmatrix}.\
\end{equation*} 
Let $u$ denote the number of diagonal blocks $J_{t+k-s}(\lambda_1)$ in $N$. Observe that $t-k-u(t+k-s)\ge 0$. Then
\begin{equation*}
\begin{aligned}
    t-k&\ge u(t+k-s)=\frac{u}{2}(t+k-s)+\frac{u}{2}(t+k-s)\\
    &\ge \frac{u}{2}(t+k-s)+\frac{u}{2}s=\frac{u}{2}(t+k),
\end{aligned}
\end{equation*}
since $t+k-s\ge s$. Thus,
\begin{equation*}
\begin{aligned}    
u\le \frac{2(t-k)}{t+k}=\frac{2(t+k)-4k}{t+k}=2-\frac{4k}{t+k}<2.
\end{aligned}
\end{equation*}
Therefore, $N$ contains at most one diagonal block $J_{t+k-s}(\lambda_1)$. 

$\bullet$ Assume that there is only one diagonal block $J_{t+k-s}(\lambda_1)$ in $N$. In this case, the Jordan block $J_{s}(\lambda_2)$ cannot coexist with $J_{t+k-s}(\lambda_1)$ in $N$, since $(t+k-s)+s=t+k>t-k$. Consequently,
\begin{equation*}
(P^{-1}AP-\lambda_1\mathnormal{E})^{t+k-s}(P^{-1}AP-\lambda_2\mathnormal{E})^{s-1}=(\lambda_2-\lambda_1)^{t+k-s}\mathnormal{E}_{t+k-s+1,t+k},
\end{equation*}
which translates to 
\begin{equation*}
(A-\lambda_1\mathnormal{E})^{t+k-s}(A-\lambda_2\mathnormal{E})^{s-1}=(\lambda_2-\lambda_1)^{t+k-s}P\mathnormal{E}_{t+k-s+1,t+k}P^{-1}. 
\end{equation*}

$\bullet$ Assume that there is no diagonal block $J_{t+k-s}(\lambda_1)$ in $N$. In this case, $P^{-1}AP$ has an eigenvalue $\lambda_1$ with a single corresponding Jordan block of the maximal size. Therefore, 
\begin{equation*}
(P^{-1}AP-\lambda_1\mathnormal{E})^{t+k-s-1}(P^{-1}AP-\lambda_2\mathnormal{E})^s=(\lambda_1-\lambda_2)^s\mathnormal{E}_{1,t+k-s},
\end{equation*}
implying
\begin{equation*}
(A-\lambda_1\mathnormal{E})^{t+k-s-1}(A-\lambda_2\mathnormal{E})^s=(\lambda_1-\lambda_2)^{s}P\mathnormal{E}_{1,t+k-s}P^{-1}. 
\end{equation*}

In both cases, $\langle\mathnormal{E}, A, \cdots, A^{t+k-1}\rangle$ contains a rank-one matrix. By Corollary 3.5, 
\begin{equation*}
\ell\left(\mathcal{S}\right)\le 2n+t+k-1-4 \le 2n+2t-5=3n-5.
\end{equation*}

\item [({3})]Assume that $A$ has three distinct eigenvalues $\lambda_1, \lambda_2$ and $\lambda_3$. The possible form of the minimal polynomial is $m(\lambda)=(\lambda-\lambda_1)^{i_1}(\lambda-\lambda_2)^{i_2}(\lambda-\lambda_3)^{i_3}$($\lambda_1$, $\lambda_2$ and $\lambda_3$ appear symmetrically in this situation), where $1\le i_3\le i_2 \le i_1 \le t+k-2$. 
\item [(\rmnum{1})]If the minimal polynomial of $A$ is $m(\lambda)=(\lambda-\lambda_1)^{t+k-j_1-1}(\lambda-\lambda_2)^{j_1}(\lambda-\lambda_3)$, where $t+k-j_1-1\ge j_1\ge 1$, by the Jordan Decomposition Theorem, then there exists an invertible matrix $P$ such that 
\begin{equation*}
\textbf{$P^{-1}AP$}=    
\begin{pmatrix}     
J_{t+k-j_1-1}(\lambda_1) &     \\         
&J_{j_1}(\lambda_2)\\          
&   &\lambda_3\\     
& &  & N_{(t-k)\times (t-k)}     
\end{pmatrix}.\
\end{equation*}
Let $u$ denote the number of the Jordan blocks $J_{t+k-j_1-1}(\lambda_1)$ in $N$. Note that $t-k-u(t+k-j_1-1)\ge 0$. Then
\begin{equation*}
\begin{aligned}    
t-k&\ge u(t+k-j_1-1)=\frac{u}{2}(t+k-j_1-1)+\frac{u}{2}(t+k-j_1-1)\\    
&\ge \frac{u}{2}(t+k-j_1-1)+\frac{u}{2}j_1=\frac{u}{2}(t+k-1),
\end{aligned}
\end{equation*}
since $t+k-j_1-1\ge j_1$. Thus,
\begin{equation*}
\begin{aligned}    
u\le \frac{2(t-k)}{t+k-1}=\frac{2(t+k-1)+2-4k}{t+k-1}=2+\frac{2-4k}{t+k-1}<2.
\end{aligned}
\end{equation*}
Therefore, there is at most one block $J_{t+k-j_1-1}(\lambda_1)$ in $N$. 

$\bullet$ Assume that there is only one diagonal block $J_{t+k-j_1-1}(\lambda_1)$ in $N$. In this case, the Jordan block $J_{j_1}(\lambda_2)$ does not exist in $N$, since $(t+k-j_1-1)+j_1=t+k-1>t-k$. Consequently,
\begin{equation*}
\begin{aligned}
&(P^{-1}AP-\lambda_1\mathnormal{E})^{t+k-j_1-1}(P^{-1}AP-\lambda_2\mathnormal{E})^{j_1-1}(P^{-1}AP-\lambda_3\mathnormal{E})\\
&=(\lambda_2-\lambda_1)^{t+k-j_1-1}(\lambda_2-\lambda_3)\mathnormal{E}_{t+k-j_1,t+k-1},
\end{aligned}
\end{equation*}
which implies
\begin{equation*}
\begin{aligned}
&(A-\lambda_1\mathnormal{E})^{t+k-j_1-1}(A-\lambda_2\mathnormal{E})^{j_1-1}(A-\lambda_3\mathnormal{E})\\
&=(\lambda_2-\lambda_1)^{t+k-j_1-1}(\lambda_2-\lambda_3)P\mathnormal{E}_{t+k-j_1,t+k-1}P^{-1}. 
\end{aligned}
\end{equation*}

$\bullet$ Assume that there is no Jordan block $J_{t+k-j_1-1}(\lambda_1)$ in $N$. In this case, $P^{-1}AP$ has an eigenvalue $\lambda_1$ which corresponds to a single maximal Jordan block. Therefore, 
\begin{equation*}
\begin{aligned}
&(P^{-1}AP-\lambda_1\mathnormal{E})^{t+k-j_1-2}(P^{-1}AP-\lambda_2\mathnormal{E})^{j_1}(P^{-1}AP-\lambda_3\mathnormal{E})\\
&=(\lambda_1-\lambda_2)^{j_1}(\lambda_1-\lambda_3)\mathnormal{E}_{1,t+k-j_1-1},
\end{aligned}
\end{equation*}
translating to 
\begin{equation*}
\begin{aligned}
&(A-\lambda_1\mathnormal{E})^{t+k-j_1-2}(A-\lambda_2\mathnormal{E})^{j_1}(A-\lambda_3\mathnormal{E})\\
&=(\lambda_1-\lambda_2)^{j_1}(\lambda_1-\lambda_3)P\mathnormal{E}_{1,t+k-j_1-1}P^{-1}. 
\end{aligned}
\end{equation*}

In both cases, the linear span $\langle\mathnormal{E}, A, \cdots, A^{t+k-1}\rangle$ contains a rank-one matrix. This implies that $\ell\left(\mathcal{S}\right)\le 2n+t+k-1-4 \le 2n+2t-5=3n-5$ by Corollary 3.5.

\item [(\rmnum{2})]If the minimal polynomial of $A$ is $m(\lambda)=(\lambda-\lambda_1)^{t+k-j_2-2}(\lambda-\lambda_2)^{j_2}(\lambda-\lambda_3)^2$, where $t+k-j_2-2\ge j_2\ge 2$ and $t+k\ge 6$, by the Jordan Decomposition Theorem, then there exists an invertible matrix $P$ such that 
\begin{equation*}
\textbf{$P^{-1}AP$}=    
\begin{pmatrix}     
J_{t+k-j_2-2}(\lambda_1) &     \\         
&J_{j_2}(\lambda_2)\\          
&   &J_2(\lambda_3)\\     
& &  & N_{(t-k)\times (t-k)}     
\end{pmatrix}.\
\end{equation*}
Let $u$ denote the number of the Jordan blocks $J_{t+k-j_2-2}(\lambda_1)$ in $N$. Note that 
\begin{equation*}
t-k-u(t+k-j_2-2)\ge 0. 
\end{equation*}
Then,
\begin{equation*}
\begin{aligned}    
t-k&\ge u(t+k-j_2-2)=\frac{u}{2}(t+k-j_2-2)+\frac{u}{2}(t+k-j_2-2)\\    
&\ge \frac{u}{2}(t+k-j_2-2)+\frac{u}{2}j_2=\frac{u}{2}(t+k-2),
\end{aligned}
\end{equation*}
since $t+k-j_2-2\ge j_2$. Thus,
\begin{equation*}
\begin{aligned}    
u\le \frac{2(t-k)}{t+k-2}=\frac{2(t+k-2)+4-4k}{t+k-2}=2+\frac{4-4k}{t+k-2}\le 2.
\end{aligned}
\end{equation*}
Therefore, $N$ contains at most two Jordan blocks $J_{t+k-j_2-2}(\lambda_1)$. 

\ding{192}If $k = 1$, then $u\le2$, that is, there are at most two Jordan blocks $J_{t-j_2-1}(\lambda_1)$ in $N$. Consider three cases as follows.
\item[(a)]If there are two diagonal blocks $J_{t-j_2-1}(\lambda_1)$ in $N$, then 
\begin{equation*}
    t-1-2(t-j_2-1)=2j_2+1-t\le 2j_2-4<2j_2,
\end{equation*}
since $t+k=t+1\ge 6$ implies $t\ge 5$. In this configuration, $N$ contains at most one Jordan block $J_{j_2}(\lambda_2)$. 
\item[$(a_1)$]Assume that there is only one Jordan block $J_{j_2}(\lambda_2)$ in $N$. Then $N$ has no Jordan block $J_2(\lambda_3)$, since $2(t-j_2-1)+j_2+2=(t-j_2-1)+t-j_2-1+j_2+2=(t-j_2-1)+t+1>t-1$. Consequently, 
\begin{equation*}
\begin{aligned}
&(P^{-1}AP-\lambda_1\mathnormal{E})^{t-j_2-1}(P^{-1}AP-\lambda_2\mathnormal{E})^{j_2}(P^{-1}AP-\lambda_3\mathnormal{E})\\
&=(\lambda_3-\lambda_1)^{t-j_2-1}(\lambda_3-\lambda_2)^{j_2}\mathnormal{E}_{t,t+1},
\end{aligned}
\end{equation*}
which implies
\begin{equation*}
\begin{aligned}
&(A-\lambda_1\mathnormal{E})^{t-j_2-1}(A-\lambda_2\mathnormal{E})^{j_2}(A-\lambda_3\mathnormal{E})\\
&=(\lambda_3-\lambda_1)^{t-j_2-1}(\lambda_3-\lambda_2)^{j_2}P\mathnormal{E}_{t,t+1}P^{-1}. 
\end{aligned}
\end{equation*}

\item[$(a_2)$]Assume that there is no jordan block $J_{j_2}(\lambda_2)$ in $N$. Here
\begin{equation*}
\begin{aligned}
&(P^{-1}AP-\lambda_1\mathnormal{E})^{t-j_2-1}(P^{-1}AP-\lambda_2\mathnormal{E})^{j_2-1}(P^{-1}AP-\lambda_3\mathnormal{E})^2\\
&=(\lambda_2-\lambda_1)^{t-j_2-1}(\lambda_2-\lambda_3)^{2}\mathnormal{E}_{t-j_2,t-1},
\end{aligned}
\end{equation*}
translating to
\begin{equation*}
\begin{aligned}
&(A-\lambda_1\mathnormal{E})^{t-j_2-1}(A-\lambda_2\mathnormal{E})^{j_2-1}(A-\lambda_3\mathnormal{E})^2\\
&=(\lambda_2-\lambda_1)^{t-j_2-1}(\lambda_2-\lambda_3)^2P\mathnormal{E}_{m-j_2,t-1}P^{-1}. 
\end{aligned}
\end{equation*}

In both cases, the linear span of $\mathnormal{E}, A, \cdots, A^{t}$ contains a rank-one matrix, so 
\begin{equation*}
\ell\left(\mathcal{S}\right)\le 2n+t+k-1-4 \le 2n+t-4=\frac{5n}{2}-4
\end{equation*}
by Corollary 3.5.

\item[(b)]If there is only one diagonal block $J_{t-j_2-1}(\lambda_1)$ in $N$, then $t-1-(t-j_2-1)=j_2$. In this case, the Jordan block $J_{j_2}(\lambda_2)$ has at most one in $N$. 
\item[$(b_1)$]Assume that there is only one block $J_{j_2}(\lambda_2)$ in $N$. Then $N$ has no block $J_2(\lambda_3)$, since $(t-j_2-1)+j_2+2=t+1>t-1$. Thus,
\begin{equation*}
\begin{aligned}
&(P^{-1}AP-\lambda_1\mathnormal{E})^{t-j_2-1}(P^{-1}AP-\lambda_2\mathnormal{E})^{j_2}(P^{-1}AP-\lambda_3\mathnormal{E})\\
&=(\lambda_3-\lambda_1)^{t-j_2-1}(\lambda_3-\lambda_2)^{j_2}\mathnormal{E}_{t,t+1},
\end{aligned}
\end{equation*}
which implies
\begin{equation*}
\begin{aligned}
&(A-\lambda_1\mathnormal{E})^{t-j_2-1}(A-\lambda_2\mathnormal{E})^{j_2}(A-\lambda_3\mathnormal{E})\\
&=(\lambda_3-\lambda_1)^{t-j_2-1}(\lambda_3-\lambda_2)^{j_2}P\mathnormal{E}_{t,t+1}P^{-1}. 
\end{aligned}
\end{equation*}
\item[$(b_2)$]Assume that there is no Jordan block $J_{j_2}(\lambda_2)$ in $N$. Here, we derive: 
\begin{equation*}
\begin{aligned}
&(P^{-1}AP-\lambda_1\mathnormal{E})^{t-j_2-1}(P^{-1}AP-\lambda_2\mathnormal{E})^{j_2-1}(P^{-1}AP-\lambda_3\mathnormal{E})^2\\
&=(\lambda_2-\lambda_1)^{t-j_2-1}(\lambda_2-\lambda_3)^{2}\mathnormal{E}_{t-j_2,t-1},
\end{aligned}
\end{equation*}
which implies 
\begin{equation*}
\begin{aligned}
&(A-\lambda_1\mathnormal{E})^{t-j_2-1}(A-\lambda_2\mathnormal{E})^{j_2-1}(A-\lambda_3\mathnormal{E})^2\\
&=(\lambda_2-\lambda_1)^{t-j_2-1}(\lambda_2-\lambda_3)^2P\mathnormal{E}_{t-j_2,t-1}P^{-1}. 
\end{aligned}
\end{equation*}
In both cases, $\langle\mathnormal{E}, A, \cdots, A^{t}\rangle$ contains a rank-one matrix. By Corollary 3.5,
\begin{equation*}
\ell\left(\mathcal{S}\right)\le 2n+t+k-1-4 \le 2n+t-4=\frac{5n}{2}-4.
\end{equation*}

\item[(c)]If there is no diagonal block $J_{t-j_2-1}(\lambda_1)$ in $N$, then we have 
\begin{equation*}
\begin{aligned}
&(P^{-1}AP-\lambda_1\mathnormal{E})^{t-j_2-2}(P^{-1}AP-\lambda_2\mathnormal{E})^{j_2}(P^{-1}AP-\lambda_3\mathnormal{E})^2\\
&=(\lambda_1-\lambda_2)^{j_2}(\lambda_1-\lambda_3)^2\mathnormal{E}_{1,t-j_2-1},
\end{aligned}
\end{equation*}
translating to
\begin{equation*}
\begin{aligned}
&(A-\lambda_1\mathnormal{E})^{t-j_2-2}(A-\lambda_2\mathnormal{E})^{j_2}(A-\lambda_3\mathnormal{E})^2\\
&=(\lambda_1-\lambda_2)^{j_2}(\lambda_1-\lambda_3)^{2}P\mathnormal{E}_{1,t-j_2-1}P^{-1}. 
\end{aligned}
\end{equation*}

Since the linear combinations of $\mathnormal{E}, A, \cdots, A^{t}$ include a rank-one matrix, Corollary 3.5 yields
\begin{equation*}
\ell\left(\mathcal{S}\right)\le 2n+t+k-1-4 \le 2n+t-4=\frac{5n}{2}-4.
\end{equation*}

\ding{193}If $k>1$, then $u\le1$. It means that there is at most one Jordan block $J_{t+k-j_2-2}(\lambda_1)$ in $N$. Consider two cases as follows.

$\bullet$ If there is one diagonal block $J_{t+k-j_2-2}(\lambda_1)$ in $N$, then the Jordan block $J_{j_2}(\lambda_2)$ not exists in $N$, since
\begin{equation*}    
t+k-j_2-2+j_2=t+k-2>t-k.
\end{equation*}
Consequently,
\begin{equation*}
\begin{aligned}
&(P^{-1}AP-\lambda_1\mathnormal{E})^{t+k-j_2-2}(P^{-1}AP-\lambda_2\mathnormal{E})^{j_2-1}(P^{-1}AP-\lambda_3\mathnormal{E})^2\\
&=(\lambda_2-\lambda_1)^{t+k-j_2-2}(\lambda_2-\lambda_3)^{2}\mathnormal{E}_{t+k-j_2-1,t+k-2},
\end{aligned}
\end{equation*}

$\bullet$ If there is no diagonal block $J_{t+k-j_2-2}(\lambda_1)$ in $N$, then we have 
\begin{equation*}
\begin{aligned}
&(P^{-1}AP-\lambda_1\mathnormal{E})^{t+k-j_2-3}(P^{-1}AP-\lambda_2\mathnormal{E})^{j_2}(P^{-1}AP-\lambda_3\mathnormal{E})^2\\
&=(\lambda_1-\lambda_2)^{j_2}(\lambda_1-\lambda_3)^{2}\mathnormal{E}_{1,t+k-j_2-2}.
\end{aligned}
\end{equation*}
In both cases, the linear span of $\mathnormal{E}, A, \cdots, A^{t+k-1}$ contains a rank-one matrix, so 
\begin{equation*}
l\left(\mathnormal{S}\right)\le 2n+t+k-1-4 \le 2n+2t-5=3n-5
\end{equation*}
by Corollary 3.5.

\item [(\rmnum{3})]If the minimal polynomial of $A$ is $m(\lambda)=(\lambda-\lambda_1)^{t+k-j_p-p}(\lambda-\lambda_2)^{j_p}(\lambda-\lambda_3)^p$, where $t+k-j_p-p\ge j_p\ge p$ and $t+k\ge 3p$, by the Jordan Decomposition Theorem, then there exists an invertible matrix $P$ such that 
\begin{equation*}
\textbf{$P^{-1}AP$}=    
\begin{pmatrix}     
J_{t+k-j_p-p}(\lambda_1) &     \\         
&J_{j_p}(\lambda_2)\\          
&   &J_p(\lambda_3)\\     
& &  & N_{(t-k)\times (t-k)}     
\end{pmatrix}.\
\end{equation*}
Let $u$ denote the number of the Jordan blocks $J_{t+k-j_p-p}(\lambda_1)$ in $N$. Note that 
\begin{equation*}
t-k-u(t+k-j_p-p)\ge 0.
\end{equation*}
Then
\begin{equation*}
\begin{aligned}    
t-k&\ge u(t+k-j_p-p)=\frac{u}{2}(t+k-j_p-p)+\frac{u}{2}(t+k-j_p-p)\\    
&\ge \frac{u}{2}(t+k-j_p-p)+\frac{u}{2}j_p=\frac{u}{2}(t+k-p),
\end{aligned}
\end{equation*}
since $t+k-j_p-p\ge j_p$. Thus,
\begin{equation*}
\begin{aligned}    
u&\le \frac{2(t-k)}{t+k-p}=\frac{2(t+k-p)+2p-4k}{t+k-p}=2+\frac{2p-4k}{t+k-p}\\
&\le 2+\frac{2p-4k}{2p}=3-\frac{2k}{p}<3,
\end{aligned}
\end{equation*}
since $t+k\ge 3p$. Therefore, there are at most two Jordan blocks $J_{t+k-j_p-p}(\lambda_1)$ in $N$. Consider three cases as follows.

\item[(a)]If there are two diagonal blocks $J_{t+k-j_p-p}(\lambda_1)$ in $N$, then 
\begin{equation*}
\begin{aligned}
t-k-2(t+k-j_p-p)&=-t-3k+2j_p+2p=2j_p+2p-(t+k)-2k\\
&\le 2j_p+2p-3p-2k=2j_p-p-2k<2j_p,
\end{aligned}
\end{equation*}
since $t+k\ge 3p$. Thus, the Jordan block $J_{j_p}(\lambda_2)$ has at most one in $N$. 
\item[$(a_1)$]Assume that there is only one block $J_{j_p}(\lambda_2)$ in $N$. Then $N$ has no diagonal block $J_p(\lambda_3)$, since $2(t+k-j_p-p)+j_p+p=(t+k-j_p-p)+t+k>t-k$. Therefore,
\begin{equation*}
\begin{aligned}
&(P^{-1}AP-\lambda_1\mathnormal{E})^{t+k-j_p-p}(P^{-1}AP-\lambda_2\mathnormal{E})^{j_p}(P^{-1}AP-\lambda_3\mathnormal{E})^{p-1}\\
&=(\lambda_3-\lambda_1)^{t+k-j_p-p}(\lambda_3-\lambda_2)^{j_p}\mathnormal{E}_{t+k-p+1,t+k}.
\end{aligned}
\end{equation*}

\item[$(a_2)$]Assume that there is no Jordan block $J_{j_p}(\lambda_2)$ in $N$. Here 
\begin{equation*}
\begin{aligned}
&(P^{-1}AP-\lambda_1\mathnormal{E})^{t+k-j_p-p}(P^{-1}AP-\lambda_2\mathnormal{E})^{j_p-1}(P^{-1}AP-\lambda_3\mathnormal{E})^p\\
&=(\lambda_2-\lambda_1)^{t+k-j_p-p}(\lambda_2-\lambda_3)^p\mathnormal{E}_{t+k-j_p-p+1,t+k-p}.
\end{aligned}
\end{equation*}
In both cases, the linear span $\langle\mathnormal{E}, A, \cdots, A^{t+k-1}\rangle$ contains a rank-one matrix. This implies that
\begin{equation*}
\ell\left(\mathcal{S}\right)\le 2n+t+k-1-4 \le 3n-5
\end{equation*}
by Corollary 3.5.

\item[(b)]If there is one diagonal block $J_{t+k-j_p-p}(\lambda_1)$ in $N$, then 
\begin{equation*}
t-k-(t+k-j_p-p)=j_p+p-2k\le j_p+j_p-2k=2j_p-2k<2j_p,
\end{equation*}
since $j_p\ge p$. Consequently, $N$ contains at most one Jordan block $J_{j_p}(\lambda_2)$. 
\item[$(b_1)$]Assume that there is only one block $J_{j_p}(\lambda_2)$ in $N$. The block $J_p(\lambda_3)$ cannot coexist in $N$, as $(t+k-j_p-p)+j_p+p=t+k>t-k$. Thus, we have 
\begin{equation*}
\begin{aligned}
&(P^{-1}AP-\lambda_1\mathnormal{E})^{t+k-j_p-p}(P^{-1}AP-\lambda_2\mathnormal{E})^{j_p}(P^{-1}AP-\lambda_3\mathnormal{E})^{p-1}\\
&=(\lambda_3-\lambda_1)^{t+k-j_p-p}(\lambda_3-\lambda_2)^{j_p}\mathnormal{E}_{t+k-p+1,t+k}.
\end{aligned}
\end{equation*}

\item[$(b_2)$]Assume that there is no Jordan block $J_{j_p}(\lambda_2)$ in $N$. Here 
\begin{equation*}
\begin{aligned}
&(P^{-1}AP-\lambda_1\mathnormal{E})^{t+k-j_p-p}(P^{-1}AP-\lambda_2\mathnormal{E})^{j_p-1}(P^{-1}AP-\lambda_3\mathnormal{E})^p\\
&=(\lambda_2-\lambda_1)^{t+k-j_p-p}(\lambda_2-\lambda_3)^p\mathnormal{E}_{t+k-j_p-p+1,t+k-p}.
\end{aligned}
\end{equation*}
In both cases, the linear span of $\mathnormal{E}, A, \cdots, A^{t+k-1}$ contains a rank-one matrix, 
\begin{equation*}
\ell\left(\mathcal{S}\right)\le 2n+t+k-1-4 \le 3n-5
\end{equation*}
by Corollary 3.5.

\item[(c)]If there is no diagonal block $J_{t+k-j_p-p}(\lambda_1)$ in $N$, then we have 
\begin{equation*}
\begin{aligned}
&(P^{-1}AP-\lambda_1\mathnormal{E})^{t+k-j_p-p-1}(P^{-1}AP-\lambda_2\mathnormal{E})^{j_p}(P^{-1}AP-\lambda_3\mathnormal{E})^{p}\\
&=(\lambda_1-\lambda_2)^{j_p}(\lambda_1-\lambda_3)^{p}\mathnormal{E}_{1,t+k-j_p-p}.
\end{aligned}
\end{equation*}

Since the linear span of $\mathnormal{E}, A, \cdots, A^{t+k-1}$ contains a rank-one matrix, Corollary 3.5 yields 
\begin{equation*}
\ell\left(\mathnormal{S}\right)\le 2n+t+k-1-4 \le 3n-5.
\end{equation*}

\item [({4})] Assume that $A$ has $s$ distinct eigenvalues $\lambda_1, \lambda_2, \cdots,\lambda_s$, then the minimal polynomial of $A$ is only possible $m(\lambda)=(\lambda-\lambda_1)^{i_1}(\lambda-\lambda_2)^{i_2}\cdots(\lambda-\lambda_s)^{i_s}$($\lambda_1, \lambda_2, \cdots,\lambda_s$ appear symmetrically in this situation), where $i_1,i_2,\cdots, i_s\in\mathbb{Z}^+$ and $1\le i_s\le i_{s-1}\le \cdots \le i_2\le i_1\le t+k-s+1$. By the Jordan Decomposition Theorem, then there exists a invertible matrix $P$ such that 
\begin{equation*}
\textbf{$P^{-1}AP$}=    
\begin{pmatrix}     
J_{i_1}(\lambda_1) &     \\         
&J_{i_2}(\lambda_2)\\          
&   &\ddots\\     
& &  & J_{i_s}(\lambda_s) \\      
& &  & & N_{(t-k)\times (t-k)}    
\end{pmatrix}.
\end{equation*}

Assume that there are not less than two blocks $J_{i_r}$ in $P^{-1}AP$ for every $r$, $r=1,2,\cdots, s$. Then 
\begin{equation*}   
n=2t\ge 2(i_1+i_2+\cdots +i_s)=2t+2k,
\end{equation*} 
this contradicts $2t+2k>2t$ for any $k$. Thus, there exists $t$, such that there are not more than two blocks $J_{i_p}$ in $P^{-1}AP$, where $1\le p\le s$. Therefore, the linear span of $\mathnormal{E}, A, \cdots, A^{t+k-1}$ contains a rank-one matrix. By Corollary 3.5, $\ell\left(\mathcal{S}\right)\le 2n+t+k-1-4=3n-5$.
\end{proof}

In the following section, the case of $n=2t+1$ is considered.
\begin{theo}
\label{T11}
Let $\mathbb{F}$ be an algebraically closed field and let $n,t\in\mathbb{Z}^+$, $n=2t+1$. Let $\mathcal{S}$ be a generating system of $\mathrm{M}_n\left(\mathbb{F}\right)$. If $m(\mathcal{S})=t+k$, where $k=1,2,\cdots,t+1$, then $l\left(\mathcal{S}\right)\le 3n-5$.
\end{theo}

\begin{proof}
There exists a matrix $A\in \mathcal{S}$ with a minimal polynomial of degree $t+k$ because $m(\mathcal{S})=t+k$.
\item [(1)]If $A$ has only one eigenvalue $\lambda_1$, the minimal polynomial of $A$ is $m(\lambda)=(\lambda-\lambda_1)^{t+k}$. By the Jordan Decomposition Theorem, then there exists an invertible matrix $P$ such that
\begin{equation*}
\textbf{$P^{-1}AP$}=    
\begin{pmatrix}     
J_{t+k}(\lambda_1) &     \\       
& N_{(t-k+1)\times (t-k+1)}     
\end{pmatrix},\
\end{equation*}
where $N-\lambda_1E$ is a nilpotent Jordan matrix of size $t-k+1$.

The structure of $N$ is 
\begin{equation*}
N=J_{k_1}(\lambda_1)\oplus J_{k_2}(\lambda_1)\oplus \cdots \oplus J_{k_s}(\lambda_1)\oplus \lambda_1E_t,
\end{equation*}
where $t-k+1=k_1+k_2+\cdots +k_s+t$ and $t-k+1 \ge k_1 \ge k_2 \ge \cdots \ge k_s\ge 2$.

It is evident that there is at most one diagonal block $J_{t+k}(\lambda_1)$ in $P^{-1}AP$, since $t+k> t-k+1$ for all $k$. Consequently, 
\begin{equation*}
(P^{-1}AP-\lambda_1\mathnormal{E})^{t+k-1}=\mathnormal{E}_{1,t+k}.
\end{equation*}
Since the linear span of $\mathnormal{E}, A, \cdots, A^{t+k-1}$ includes the rank-one matrix $P\mathnormal{E}_{1,t+k}P^{-1}$, it follow that 
\begin{equation*}
\ell\left(\mathcal{S}\right)\le 2n+t+k-1-4 \le 2n+2t-5=3n-5,
\end{equation*}
where the first inequality holds by Corollary 3.5.

\item [(2)]If $A$ has two distinct eigenvalues $\lambda_1, \lambda_2$, then the minimal polynomial of $A$ is only possible $m(\lambda)=(\lambda-\lambda_1)^{t+k-s}(\lambda-\lambda_2)^s$($\lambda_1$ and $\lambda_2$ appear symmetrically in this situation), where $s\in\mathbb{Z}^+$ and $t+k-s\ge s$.
\item [(\rmnum{1})]If the minimal polynomial of $A$ is $m(\lambda)=(\lambda-\lambda_1)^{t+k-1}(\lambda-\lambda_2)$, by the Jordan Decomposition Theorem, then there exists an invertible matrix $P$ such that 
\begin{equation*}
\textbf{$P^{-1}AP$}=    
\begin{pmatrix}     
J_{t+k-1}(\lambda_1) &\\         
&\lambda_2\\     
&  & N_{(t-k+1)\times (t-k+1)}     
\end{pmatrix}.\
\end{equation*}
There are at most two diagonal blocks $J_{t+k-1}(\lambda_1)$ in $P^{-1}AP$, since $t+k-1\ge t-k+1$ for all $k$. 

\ding{192}If $k>1$, then $t+k-1> t-k+1$. Thus, $P^{-1}AP$ has only one diagonal block $J_{t+k-1}(\lambda_1)$. Therefore, we have 
\begin{equation*}
(P^{-1}AP-\lambda_1\mathnormal{E})^{t+k-2}(P^{-1}AP-\lambda_2\mathnormal{E})=(\lambda_1-\lambda_2)\mathnormal{E}_{1,t+k-1}.
\end{equation*}
The linear span of $\mathnormal{E}, A, \cdots, A^{t+k-1}$ contains the rank-one matrix $P\mathnormal{E}_{1,t+k-1}P^{-1}$, so 
\begin{equation*}
\ell\left(\mathcal{S}\right)\le 2n+t+k-1-4 \le 3n-5
\end{equation*}
by Corollary 3.5.

\ding{193}If $k=1$, then $t+k-1= t-k+1=t$. In this case, the Jordan form becomes 
\begin{equation*}
\textbf{$P^{-1}AP$}=    
\begin{pmatrix}     
J_{t}(\lambda_1) &     \\    
&\lambda_2\\
&  & N_{t\times t}     
\end{pmatrix}.\
\end{equation*}
There are at most two diagonal blocks $J_t(\lambda_1)$ in $P^{-1}AP$. Consider two cases as follow.

$\bullet$ Assume that there are two diagonal blocks $J_{t}(\lambda_1)$ in $P^{-1}AP$. In this case, 
\begin{equation*}
\textbf{$P^{-1}AP$}=    
\begin{pmatrix}     
J_{t}(\lambda_1) &     \\    
&\lambda_2\\
&  & J_{t}(\lambda_1)     
\end{pmatrix}.\
\end{equation*} 
Then 
\begin{equation*}
(P^{-1}AP-\lambda_1\mathnormal{E})^{t}=(\lambda_2-\lambda_1)^{t}\mathnormal{E}_{t+1,t+1}.
\end{equation*}

$\bullet$ Assume that there is at most one diagonal block $J_{t}(\lambda_1)$ in $P^{-1}AP$. In this case, $P^{-1}AP$ has an eigenvalue $\lambda_1$ with a single corresponding Jordan block of the maximal size. Consequently,
\begin{equation*}
(P^{-1}AP-\lambda_1\mathnormal{E})^{t-1}(P^{-1}AP-\lambda_2\mathnormal{E})=(\lambda_1-\lambda_2)\mathnormal{E}_{1,t}.
\end{equation*}

In both cases, the linear span of $\mathnormal{E}, A, \cdots, A^{t}$ contains a rank-one matrix. By Corollary 3.5, we conclude
\begin{equation*}
\ell\left(\mathcal{S}\right)\le 2n+t+k-1-4 \le 3n-5.
\end{equation*}

\item [(\rmnum{2})]If the minimal polynomial of $A$ is $m(\lambda)=(\lambda-\lambda_1)^{t+k-2}(\lambda-\lambda_2)^2$, by the Jordan Decomposition Theorem, then there exists an invertible matrix $P$ such that 
\begin{equation*}
\textbf{$P^{-1}AP$}=    
\begin{pmatrix}     
J_{t+k-2}(\lambda_1) &\\         
&J_2(\lambda_2)\\     
&  & N_{(t-k+1)\times (t-k+1)}     
\end{pmatrix}.\
\end{equation*}
There is at most one diagonal block $J_{t+k-2}(\lambda_1)$ in $N$. 

\ding{192}If $k>1$, then $t+k-2> t-k+1$. This implies that $P^{-1}AP$ has only one diagonal block $J_{t+k-2}(\lambda_1)$. Consequently, 
\begin{equation*}
(P^{-1}AP-\lambda_1\mathnormal{E})^{t+k-3}(P^{-1}AP-\lambda_2\mathnormal{E})^2=(\lambda_1-\lambda_2)^2\mathnormal{E}_{1,t+k-2}.
\end{equation*}
The linear span of $\mathnormal{E}, A, \cdots, A^{t+k-1}$ contains the rank-one matrix $P\mathnormal{E}_{1,t+k-2}P^{-1}$. By Corollary 3.5, $\ell\left(\mathcal{S}\right)\le 2n+t+k-1-4 \le 3n-5$.

\ding{193}If $k=1$, then $t+k-2<t-k+1$. In this case, the Jordan form becomes
\begin{equation*}
\textbf{$P^{-1}AP$}=    
\begin{pmatrix}     
J_{t-1}(\lambda_1) &     \\    
&J_2(\lambda_2)\\
&  & N_{t\times t}     
\end{pmatrix}.\
\end{equation*}
There are at most two diagonal blocks $J_{t-1}(\lambda_1)$ in $P^{-1}AP$. Consider two cases as follow.

$\bullet$ Assume that there are two diagonal blocks $J_{t-1}(\lambda_1)$ in $P^{-1}AP$. Thus, $N$ has no Jordan block $J_2(\lambda_2)$, since $2(t-1)+4=2t+2>n=2t+1$. Therefore, we have 
\begin{equation*}
(P^{-1}AP-\lambda_1\mathnormal{E})^{t-1}(P^{-1}AP-\lambda_2\mathnormal{E})=(\lambda_2-\lambda_1)^{t-1}\mathnormal{E}_{t,t+1}.
\end{equation*}

$\bullet$ Assume that there is at most one diagonal block $J_{t}(\lambda_1)$ in $P^{-1}AP$. In this case, $P^{-1}AP$ has an eigenvalue $\lambda_1$ with a single corresponding Jordan block of the maximal size. Then
\begin{equation*}
(P^{-1}AP-\lambda_1\mathnormal{E})^{t-2}(P^{-1}AP-\lambda_2\mathnormal{E})^2=(\lambda_1-\lambda_2)^2\mathnormal{E}_{1,t-1}.
\end{equation*}
In both cases, $\langle\mathnormal{E}, A, \cdots, A^{t}\rangle$ contains a rank-one matrix, so $\ell\left(\mathcal{S}\right)\le 2n+t+k-1-4 \le 3n-5$ by Corollary 3.5.

\item [(\rmnum{3})]If the minimal polynomial of $A$ is $m(\lambda)=(\lambda-\lambda_1)^{t+k-s}(\lambda-\lambda_2)^s$, where $3\le s\le \lfloor{\frac{t+k}{2}}\rfloor$, by the Jordan Decomposition Theorem, then there exists an invertible matrix $P$ such that 
\begin{equation*}
\textbf{$P^{-1}AP$}=    
\begin{pmatrix}     
J_{t+k-s}(\lambda_1) &     \\         
&J_s(\lambda_2)\\     
&  & N_{(t-k+1)\times (t-k+1)}     
\end{pmatrix}.\
\end{equation*} 
Let $u$ denote the number of diagonal blocks $J_{t+k-s}(\lambda_1)$ in $N$. Note that $t-k+1-u(t+k-s)\ge 0$. Then 
\begin{equation*}
\begin{aligned}
    t-k+1&\ge u(t+k-s)=\frac{u}{2}(t+k-s)+\frac{u}{2}(t+k-s)\\
    &\ge \frac{u}{2}(t+k-s)+\frac{u}{2}s=\frac{u}{2}(t+k),
\end{aligned}
\end{equation*}
since $t+k-s\ge s$. Thus,
\begin{equation*}
\begin{aligned}    
u\le \frac{2(t-k+1)}{t+k}=\frac{2(t+k)+2-4k}{t+k}=2+\frac{2-4k}{t+k}<2.
\end{aligned}
\end{equation*}
Therefore, there is at most one diagonal block $J_{t+k-s}(\lambda_1)$ in $N$. 

$\bullet$ Assume that there is only one diagonal block $J_{t+k-s}(\lambda_1)$ in $N$. In this case, the Jordan block $J_{s}(\lambda_2)$ does not exist in $N$, since $(t+k-s)+s=t+k>t-k+1$. Consequently,
\begin{equation*}
(P^{-1}AP-\lambda_1\mathnormal{E})^{t+k-s}(P^{-1}AP-\lambda_2\mathnormal{E})^{s-1}=(\lambda_2-\lambda_1)^{t+k-s}\mathnormal{E}_{t+k-s+1,t+k}.
\end{equation*}

$\bullet$ Assume that there is no diagonal block $J_{t+k-s}(\lambda_1)$ in $N$. In this case, $P^{-1}AP$ has an eigenvalue $\lambda_1$ with a single corresponding Jordan block of the maximal size. Thus, we have 
\begin{equation*}
(P^{-1}AP-\lambda_1\mathnormal{E})^{t+k-s-1}(P^{-1}AP-\lambda_2\mathnormal{E})^s=(\lambda_1-\lambda_2)^s\mathnormal{E}_{1,t+k-s}.
\end{equation*}

In both cases, the linear span of $\mathnormal{E}, A, \cdots, A^{t+k-1}$ contains a rank-one matrix. By Corollary 3.5, $\ell\left(\mathcal{S}\right)\le 2n+t+k-1-4 \le 2n+2t+1-5=3n-5$.

\item [({3})]Assume that $A$ has three distinct eigenvalues $\lambda_1, \lambda_2$ and $\lambda_3$. The possible form of the minimal polynomial is $m(\lambda)=(\lambda-\lambda_1)^{i_1}(\lambda-\lambda_2)^{i_2}(\lambda-\lambda_3)^{i_3}$($\lambda_1$, $\lambda_2$ and $\lambda_3$ appear symmetrically in this situation), where $1\le i_3\le i_2 \le i_1 \le t+k-2$. 
\item [(\rmnum{1})]If the minimal polynomial of $A$ is $m(\lambda)=(\lambda-\lambda_1)^{t+k-j_1-1}(\lambda-\lambda_2)^{j_1}(\lambda-\lambda_3)$, where $t+k-j_1-1\ge j_1\ge 1$, by the Jordan Decomposition Theorem, then there exists an invertible matrix $P$ such that 
\begin{equation*}
\textbf{$P^{-1}AP$}=    
\begin{pmatrix}     
J_{t+k-j_1-1}(\lambda_1) &     \\         
&J_{j_1}(\lambda_2)\\          
&   &\lambda_3\\     
& &  & N_{(t-k+1)\times (t-k+1)}     
\end{pmatrix}.\
\end{equation*}
Let $u$ denote the number of diagonal blocks $J_{t+k-j_1-1}(\lambda_1)$ in $N$. Note that 
\begin{equation*}
t-k-u(t+k-j_1-1)\ge 0.
\end{equation*}
Then
\begin{equation*}
\begin{aligned}    
t-k+1&\ge u(t+k-j_1-1)=\frac{u}{2}(t+k-j_1-1)+\frac{u}{2}(t+k-j_1-1)\\    
&\ge \frac{u}{2}(t+k-j_1-1)+\frac{u}{2}j_1=\frac{u}{2}(t+k-1),
\end{aligned}
\end{equation*}
since $t+k-j_1-1\ge j_1$. Thus,
\begin{equation*}
\begin{aligned}    
u\le \frac{2(t-k+1)}{t+k-1}=\frac{2(t+k-1)+4-4k}{t+k-1}=2+\frac{4-4k}{t+k-1}\le 2.
\end{aligned}
\end{equation*}
Therefore, there are at most two Jordan blocks $J_{t+k-j_1-1}(\lambda_1)$ in $N$. 

\item[(a)] If $k = 1$, then $u\le2$, that is, there are at most two blocks $J_{t-j_1}(\lambda_1)$ in $N$. Consider three cases as follows.

\item [$(a_1)$]If there are two diagonal blocks $J_{t-j_1}(\lambda_1)$ in $N$, then the Jordan block $J_{j_1}(\lambda_2)$ cannot coexist with $J_{t-j_1}(\lambda_1)$ in $N$, since $2(t-j_1)+j_1=(t-j_1)+t>t$. Thus, we have
\begin{equation*}
\begin{aligned}
&(P^{-1}AP-\lambda_1\mathnormal{E})^{t-j_1}(P^{-1}AP-\lambda_2\mathnormal{E})^{j_1-1}(P^{-1}AP-\lambda_3\mathnormal{E})\\
&=(\lambda_2-\lambda_1)^{t-j_1}(\lambda_2-\lambda_3)\mathnormal{E}_{t-j_1+1,t}.
\end{aligned}
\end{equation*}

\item [$(a_2)$]If there is only one diagonal block $J_{t-j_1}(\lambda_1)$ in $N$, then the Jordan block $J_{j_1}(\lambda_2)$ is at most one in $N$, since $(t-j_1)+j_1=t$. Consider two cases.

$\bullet$ Assume that there is one block $J_{j_1}(\lambda_2)$ in $N$. $N$ has no $\lambda_3E_1$, since 
\begin{equation*}
(t-j_1)+j_1+1=t+1>t.
\end{equation*}
Thus, we have
\begin{equation*}
(P^{-1}AP-\lambda_1\mathnormal{E})^{t-j_1}(P^{-1}AP-\lambda_2\mathnormal{E})^{j_1}=(\lambda_3-\lambda_1)^{t-j_1}(\lambda_3-\lambda_2)^{j_1}\mathnormal{E}_{t+1,t+1}.
\end{equation*}

$\bullet$ Assume that there is no block $J_{j_1}(\lambda_2)$ in $N$. Then,
\begin{equation*}
\begin{aligned}
&(P^{-1}AP-\lambda_1\mathnormal{E})^{t-j_1}(P^{-1}AP-\lambda_2\mathnormal{E})^{j_1-1}(P^{-1}AP-\lambda_3\mathnormal{E})\\
&=(\lambda_2-\lambda_1)^{t-j_1}(\lambda_2-\lambda_3)\mathnormal{E}_{t-j_1+1,t}.
\end{aligned}
\end{equation*}

\item[$(a_3)$]If there is no Jordan block $J_{t-j_1}(\lambda_1)$ in $N$. In this case, $P^{-1}AP$ has an eigenvalues $\lambda_1$ with a single corresponding Jordan block of the maximal size. Therefore,
\begin{equation*}
\begin{aligned}
&(P^{-1}AP-\lambda_1\mathnormal{E})^{t-j_1-1}(P^{-1}AP-\lambda_2\mathnormal{E})^{j_1}(P^{-1}AP-\lambda_3\mathnormal{E})\\
&=(\lambda_1-\lambda_2)^{j_1}(\lambda_1-\lambda_3)\mathnormal{E}_{1,t-j_1},
\end{aligned}
\end{equation*}

In above cases, $\langle\mathnormal{E}, A, \cdots, A^{t}\rangle$ contains a rank-one matrix. By Corollary 3.5, we conclude $\ell\left(\mathcal{S}\right)\le 2n+t+k-1-4 \le 2n+t-4=\frac{5n}{2}-\frac{9}{2}$.

\item[(b)] If $k>1$, then $u<2$. It implies that there is at most one Jordan block $J_{t+k-j_1-1}(\lambda_1)$ in $N$. Consider two cases as follows.

$\bullet$ Assume that there is only one diagonal block $J_{t+k-j_1-1}(\lambda_1)$ in $N$. In this case, the Jordan block $J_{j_1}(\lambda_2)$ does not exist in $N$, since $(t+k-j_1-1)+j_1=t+k-1>t-k+1$. Thus, we have 
\begin{equation*}
\begin{aligned}
&(P^{-1}AP-\lambda_1\mathnormal{E})^{t+k-j_1-1}(P^{-1}AP-\lambda_2\mathnormal{E})^{j_1-1}(P^{-1}AP-\lambda_3\mathnormal{E})\\
&=(\lambda_2-\lambda_1)^{t+k-j_1-1}(\lambda_2-\lambda_3)\mathnormal{E}_{t+k-j_1,t+k-1}.
\end{aligned}
\end{equation*}

$\bullet$ Assume that there is no diagonal block $J_{t+k-j_1-1}(\lambda_1)$ in $N$. In this case, $P^{-1}AP$ has an eigenvalue $\lambda_1$ with a single corresponding Jordan block of the maximal size. Thus, we have 
\begin{equation*}
\begin{aligned}
&(P^{-1}AP-\lambda_1\mathnormal{E})^{t+k-j_1-2}(P^{-1}AP-\lambda_2\mathnormal{E})^{j_1}(P^{-1}AP-\lambda_3\mathnormal{E})\\
&=(\lambda_1-\lambda_2)^{j_1}(\lambda_1-\lambda_3)\mathnormal{E}_{1,t+k-j_1-1}.
\end{aligned}
\end{equation*}
In both cases, the linear span of $\mathnormal{E}, A, \cdots, A^{t+k-1}$ contains a rank-one matrix. By Corollary 3.5, $\ell\left(\mathcal{S}\right)\le 2n+t+k-1-4 \le 2n+2t+1-5=3n-5$.

\item [(\rmnum{2})]If the minimal polynomial of $A$ is $m(\lambda)=(\lambda-\lambda_1)^{t+k-j_2-2}(\lambda-\lambda_2)^{j_2}(\lambda-\lambda_3)^2$, where $t+k-j_2-2\ge j_2\ge 2$ and $t+k\ge 6$, by the Jordan Decomposition Theorem, then there exists an invertible matrix $P$ such that 
\begin{equation*}
\textbf{$P^{-1}AP$}=    
\begin{pmatrix}     
J_{t+k-j_2-2}(\lambda_1) &     \\         
&J_{j_2}(\lambda_2)\\          
&   &J_2(\lambda_3)\\     
& &  & N_{(t-k+1)\times (t-k+1)}     
\end{pmatrix}.\
\end{equation*}
Let $u$ denote the number of the Jordan blocks $J_{t+k-j_2-2}(\lambda_1)$ in $N$. Note that 
\begin{equation*}
t-k+1-u(t+k-j_2-2)\ge 0.
\end{equation*}
Then
\begin{equation*}
\begin{aligned}    
t-k+1&\ge u(t+k-j_2-2)=\frac{u}{2}(t+k-j_2-2)+\frac{u}{2}(t+k-j_2-2)\\    
&\ge \frac{u}{2}(t+k-j_2-2)+\frac{u}{2}j_2=\frac{u}{2}(t+k-2),
\end{aligned}
\end{equation*}
since $t+k-j_2-2\ge j_2$. Thus,
\begin{equation*}
\begin{aligned}    
u\le \frac{2(t-k+1)}{t+k-2}=\frac{2(t+k-2)+6-4k}{t+k-2}=2+\frac{6-4k}{t+k-2}.
\end{aligned}
\end{equation*}
Consider two cases as follows.
\item[(a)] If $k = 1$, then $u\le 2+\frac{2}{t-1}\le 2+\frac{2}{4}<3$, that is, there are at most two blocks $J_{t-j_2-1}(\lambda_1)$ in $N$. Consider three cases as follows.
\item [$(a_1)$]If there are two diagonal blocks $J_{t-j_2-1}(\lambda_1)$ in $N$, then the block $J_{j_2}(\lambda_2)$ does not exist in $N$, since $2(t-j_2-1)+j_2=(t-j_2-1)+t-1\ge 2+t-1=t+1>t$. Consequently,
\begin{equation*}
\begin{aligned}
&(P^{-1}AP-\lambda_1\mathnormal{E})^{t-j_2-1}(P^{-1}AP-\lambda_2\mathnormal{E})^{j_2-1}(P^{-1}AP-\lambda_3\mathnormal{E})^2\\
&=(\lambda_2-\lambda_1)^{t-j_2-1}(\lambda_2-\lambda_3)^2\mathnormal{E}_{t-j_2,t-1}.
\end{aligned}
\end{equation*}

\item [$(a_2)$]If there is only one diagonal block $J_{t-j_2-1}(\lambda_1)$ in $N$, then the Jordan block $J_{j_2}(\lambda_2)$ is at most one in $N$, since $t-k+1-(t-j_2-1)-j_2=j_2+2-k\le j_2+1<2j_2$. Consider two cases.

$\bullet$ Assume that there is one block $J_{j_2}(\lambda_2)$ in $N$. $N$ has no $J_2(\lambda_3)$, since 
\begin{equation*}
(t-j_2-1)+j_2+2=t+1>t.
\end{equation*}
Then
\begin{equation*}
\begin{aligned}
&(P^{-1}AP-\lambda_1\mathnormal{E})^{t-j_2-1}(P^{-1}AP-\lambda_2\mathnormal{E})^{j_2}(P^{-1}AP-\lambda_3\mathnormal{E})\\
&=(\lambda_3-\lambda_1)^{t-j_2-1}(\lambda_3-\lambda_2)^{j_2}\mathnormal{E}_{t,t+1}.
\end{aligned}
\end{equation*}

$\bullet$ Assume that there is no block $J_{j_2}(\lambda_2)$ in $N$. Consequently,
\begin{equation*}
\begin{aligned}
&(P^{-1}AP-\lambda_1\mathnormal{E})^{t-j_2-1}(P^{-1}AP-\lambda_2\mathnormal{E})^{j_2-1}(P^{-1}AP-\lambda_3\mathnormal{E})^2\\
&=(\lambda_2-\lambda_1)^{t-j_2-1}(\lambda_2-\lambda_3)^2\mathnormal{E}_{t-j_2,t-1}.
\end{aligned}
\end{equation*}

\item[$(a_3)$]If there is no Jordan block $J_{t-j_2-1}(\lambda_1)$ in $N$. Thus, we have 
\begin{equation*}
\begin{aligned}
&(P^{-1}AP-\lambda_1\mathnormal{E})^{t-j_2-2}(P^{-1}AP-\lambda_2\mathnormal{E})^{j_2}(P^{-1}AP-\lambda_3\mathnormal{E})^2\\
&=(\lambda_1-\lambda_2)^{j_2}(\lambda_1-\lambda_3)^2\mathnormal{E}_{1,t-j_2-1}.
\end{aligned}
\end{equation*}
In above cases, the linear span of $\mathnormal{E}, A, \cdots, A^{t+k-1}$ contains a rank-one matrix. By Corollary 3.5, $\ell\left(\mathcal{S}\right)\le 2n+t+k-1-4 \le 3n-5$.

\item[(b)] If $k>1$, then $u\le 2+\frac{6-4k}{t+k-2}<2$, that is, there is at most one Jordan block $J_{t+k-j_2-2}(\lambda_1)$ in $N$. Consider two cases as follows.

\item[$(b_1)$]Assume that there is only one block $J_{t+k-j_2-2}(\lambda_1)$ in $N$. Then $N$ has no block $J_{j_2}(\lambda_2)$, since $(t+k-j_2-2)+j_2=t+k-2\ge t>t-k+1$. Thus, we have 
\begin{equation*}
\begin{aligned}
&(P^{-1}AP-\lambda_1\mathnormal{E})^{t+k-j_2-2}(P^{-1}AP-\lambda_2\mathnormal{E})^{j_2-1}(P^{-1}AP-\lambda_3\mathnormal{E})^2\\
&=(\lambda_2-\lambda_1)^{t+k-j_2-2}(\lambda_2-\lambda_3)^2\mathnormal{E}_{t+k-j_2-1,t+k-2}.
\end{aligned}
\end{equation*}

\item[$(b_2)$]Assume that there is no block $J_{t+k-j_2-2}(\lambda_1)$ in $N$. Then 
\begin{equation*}
\begin{aligned}
&(P^{-1}AP-\lambda_1\mathnormal{E})^{t+k-j_2-3}(P^{-1}AP-\lambda_2\mathnormal{E})^{j_2}(P^{-1}AP-\lambda_3\mathnormal{E})^2\\
&=(\lambda_1-\lambda_2)^{j_2}(\lambda_1-\lambda_3)^2\mathnormal{E}_{1, t+k-j_2-2}.
\end{aligned}
\end{equation*}

In both cases, the linear span of $\mathnormal{E}, A, \cdots, A^{t+k-1}$ contains a rank-one matrix. By Corollary 3.5, we conclude $\ell\left(\mathcal{S}\right)\le 2n+t+k-1-4 \le 3n-5$.

\item [(\rmnum{3})]If the minimal polynomial of $A$ is $m(\lambda)=(\lambda-\lambda_1)^{t+k-j_p-p}(\lambda-\lambda_2)^{j_p}(\lambda-\lambda_3)^p$, where $t+k-j_p-p\ge j_p\ge p$ and $t+k\ge 3p$, by the Jordan Decomposition Theorem, then there exists an invertible matrix $P$ such that 
\begin{equation*}
\textbf{$P^{-1}AP$}=    
\begin{pmatrix}     
J_{t+k-j_p-p}(\lambda_1) &     \\         
&J_{j_p}(\lambda_2)\\          
&   &J_p(\lambda_3)\\     
& &  & N_{(t-k+1)\times (t-k+1)}     
\end{pmatrix}.\
\end{equation*}
Let $u$ denote the number of the Jordan blocks $J_{t+k-j_p-p}(\lambda_1)$ in $N$. Note that 
\begin{equation*}
t-k+1-u(t+k-j_p-p)\ge 0.
\end{equation*}
Then
\begin{equation*}
\begin{aligned}    
t-k+1&\ge u(t+k-j_p-p)=\frac{u}{2}(t+k-j_p-p)+\frac{u}{2}(t+k-j_p-p)\\    
&\ge \frac{u}{2}(t+k-j_p-p)+\frac{u}{2}j_p=\frac{u}{2}(t+k-p),
\end{aligned}
\end{equation*}
since $t+k-j_p-p\ge j_p$. Thus,
\begin{equation*}
\begin{aligned}    
u&\le \frac{2(t-k+1)}{t+k-p}=\frac{2(t+k-p)+2p+2-4k}{t+k-p}=2+\frac{2p+2-4k}{t+k-p}\\
&\le 2+\frac{2p+2-4k}{2p}=3+\frac{2-4k}{2p}<3,
\end{aligned}
\end{equation*}
since $t+k\ge 3p$. Therefore, there are at most two blocks $J_{t+k-j_p-p}(\lambda_1)$ in $N$. Consider three cases as follows.

\item[(a)]If there are two diagonal blocks $J_{t+k-j_p-p}(\lambda_1)$ in $N$, then 
\begin{equation*}
\begin{aligned}
t-k+1-2(t+k-j_p-p)&=-t-3k+1+2j_p+2p=2j_p+2p-(t+k)-2k+1\\
&\le 2j_p+2p-3p-2k+1=2j_p-p-2k+1<2j_p,
\end{aligned}
\end{equation*}
since $t+k\ge 3p$. Thus, the Jordan block $J_{j_p}(\lambda_2)$ has at most one in $N$. 
\item[$(a_1)$]Assume that there is only one block $J_{j_p}(\lambda_2)$ in $N$. Then $N$ has no block $J_p(\lambda_3)$, since $2(t+k-j_p-p)+j_p+p=(t+k-j_p-p)+t+k>t-k+1$. Then 
\begin{equation*}
\begin{aligned}
&(P^{-1}AP-\lambda_1\mathnormal{E})^{t+k-j_p-p}(P^{-1}AP-\lambda_2\mathnormal{E})^{j_p}(P^{-1}AP-\lambda_3\mathnormal{E})^{p-1}\\
&=(\lambda_3-\lambda_1)^{t+k-j_p-p}(\lambda_3-\lambda_2)^{j_p}\mathnormal{E}_{t+k-p+1,t+k}.
\end{aligned}
\end{equation*}
\item[$(a_2)$]Assume that there is no Jordan block $J_{j_p}(\lambda_2)$ in $N$. Thus, we have 
\begin{equation*}
\begin{aligned}
&(P^{-1}AP-\lambda_1\mathnormal{E})^{t+k-j_p-p}(P^{-1}AP-\lambda_2\mathnormal{E})^{j_p-1}(P^{-1}AP-\lambda_3\mathnormal{E})^p\\
&=(\lambda_2-\lambda_1)^{t+k-j_p-p}(\lambda_2-\lambda_3)^p\mathnormal{E}_{t+k-j_p-p+1,t+k-p}.
\end{aligned}
\end{equation*}

In both cases, the linear span of $\mathnormal{E}, A, \cdots, A^{t+k-1}$ contains a rank-one matrix. By Corollary 3.5, $\ell\left(\mathcal{S}\right)\le 2n+t+k-1-4 \le 3n-5$.

\item[(b)]If there is one diagonal block $J_{t+k-j_p-p}(\lambda_1)$ in $N$, then 
\begin{equation*}
t-k+1-(t+k-j_p-p)=j_p+p-2k+1\le j_p+j_p-2k+1=2j_p-2k+1<2j_p,
\end{equation*}
since $j_p\ge p$. It implies that the Jordan block $J_{j_p}(\lambda_2)$ has at most one in $N$. 
\item[$(b_1)$]Assume that there is only one block $J_{j_p}(\lambda_2)$ in $N$. Then $N$ has no Jordan block $J_p(\lambda_3)$, since $(t+k-j_p-p)+j_p+p=t+k>t-k+1$. Thus, we have 
\begin{equation*}
\begin{aligned}
&(P^{-1}AP-\lambda_1\mathnormal{E})^{t+k-j_p-p}(P^{-1}AP-\lambda_2\mathnormal{E})^{j_p}(P^{-1}AP-\lambda_3\mathnormal{E})^{p-1}\\
&=(\lambda_3-\lambda_1)^{t+k-j_p-p}(\lambda_3-\lambda_2)^{j_p}\mathnormal{E}_{t+k-p+1,t+k}.
\end{aligned}
\end{equation*}
\item[$(b_2)$]Assume that there is no Jordan block $J_{j_p}(\lambda_2)$ in $N$. Consequently, 
\begin{equation*}
\begin{aligned}
&(P^{-1}AP-\lambda_1\mathnormal{E})^{t+k-j_p-p}(P^{-1}AP-\lambda_2\mathnormal{E})^{j_p-1}(P^{-1}AP-\lambda_3\mathnormal{E})^p\\
&=(\lambda_2-\lambda_1)^{t+k-j_p-p}(\lambda_2-\lambda_3)^p\mathnormal{E}_{t+k-j_p-p+1,t+k-p}.
\end{aligned}
\end{equation*}
In both cases, $\langle\mathnormal{E}, A, \cdots, A^{t+k-1}\rangle$ contains a rank-one matrix. By Corollary 3.5, 
\begin{equation*}
\ell\left(\mathcal{S}\right)\le 2n+t+k-1-4 \le 3n-5.
\end{equation*}

\item[(c)]If there is no diagonal block $J_{t+k-j_p-p}(\lambda_1)$ in $N$, then we have 
\begin{equation*}
\begin{aligned}
&(P^{-1}AP-\lambda_1\mathnormal{E})^{t+k-j_p-p-1}(P^{-1}AP-\lambda_2\mathnormal{E})^{j_p}(P^{-1}AP-\lambda_3\mathnormal{E})^{p}\\
&=(\lambda_1-\lambda_2)^{j_p}(\lambda_1-\lambda_3)^{p}\mathnormal{E}_{1,t+k-j_p-p}.
\end{aligned}
\end{equation*}

In this case, the linear span of $\mathnormal{E}, A, \cdots, A^{t+k-1}$ contains a rank-one matrix. By Corollary 3.5, $\ell\left(\mathcal{S}\right)\le 2n+t+k-1-4 \le 3n-5$.

\item [({4})] Assume that $A$ has $s$ distinct eigenvalues $\lambda_1, \lambda_2, \cdots,\lambda_s$, where $4\le s\le t$ and $s\in \mathbb{N}$. Note that $\lambda_1, \lambda_2, \cdots,\lambda_s$ appear symmetrically, then the minimal polynomial of $A$ is only possible 
\begin{equation*}
m(\lambda)=(\lambda-\lambda_1)^{i_1}(\lambda-\lambda_2)^{i_2}\cdots(\lambda-\lambda_s)^{i_s},
\end{equation*}
where $i_1,i_2,\cdots, i_s\in\mathbb{Z}^+$ and $1\le i_s\le i_{s-1}\le \cdots \le i_2\le i_1\le t+k-s+1.$
By the Jordan Decomposition Theorem, then there exists a invertible matrix $P$ such that 
\begin{equation*}
\textbf{$P^{-1}AP$}=    
\begin{pmatrix}     
J_{i_1}(\lambda_1) &     \\         
&J_{i_2}(\lambda_2)\\          
&   &\ddots\\     
& &  & J_{i_s}(\lambda_s) \\      
& &  & & N_{(t-k)\times (t-k)}    
\end{pmatrix}.
\end{equation*}
Assume that there are not less than two blocks $J_{i_r}$ in $P^{-1}AP$ for every $r$, $r=1,2,\cdots, s$. Then 
\begin{equation*}   
n=2t\ge 2(i_1+i_2+\cdots +i_s)=2t+2k,
\end{equation*} 
this contradicts $2t+2k>2t$ for any $k$. Thus, there exists $q$, such that there are not more than two blocks $J_{i_q}$ in $P^{-1}AP$, where $1\le q\le s$. Therefore, the linear span of $\mathnormal{E}, A, \cdots, A^{t+k-1}$ contains a rank-one matrix, so $\ell\left(\mathcal{S}\right)\le 2n+t+k-1-4=3n-5$ in this case.
\end{proof}

\begin{theo}
\label{T12}
Let $\mathbb{F}$ be an algebraically closed field and let $n,t\in\mathbb{Z}^+$. Let $\mathnormal{S}$ be a generating system of $\mathnormal{M}_n\left(\mathbb{F}\right)$. If $2t\le n\le 3t-1$ and $m(\mathnormal{S})=t$, then $l\left(\mathnormal{S}\right)\le \frac{7n}{2}-4$.
\end{theo}

\begin{proof}
$\mathnormal{S}$ contains a matrix $A$ with minimal polynomial of degree $t$, since $m(\mathnormal{S})=t$.
Assume that $A$ has $s$ distinct eigenvalues $\lambda_1, \lambda_2, \cdots,\lambda_s$, then the minimal polynomial of $A$ is only possible $m(\lambda)=(\lambda-\lambda_1)^{i_1}(\lambda-\lambda_2)^{i_2}\cdots(\lambda-\lambda_s)^{i_s}$($\lambda_1, \lambda_2, \cdots,\lambda_s$ appear symmetrically in this situation), where $i_1,i_2,\cdots, i_s\in\mathbb{Z}^+$ and $1\le i_s\le i_{s-1}\le \cdots \le i_2\le i_1\le t-s+1$.By the Jordan Decomposition Theorem, then there exists an invertible matrix $P$ such that 
\begin{equation*}
\textbf{$P^{-1}AP$}=
    \begin{pmatrix}
     J_{i_1}(\lambda_1) &     \\
         &J_{i_2}(\lambda_2)\\
          &   &\ddots\\
     & &  & J_{i_s}(\lambda_s) \\
      & &  & & N_{(n-t)\times (n-t)}
    \end{pmatrix}.
\end{equation*}

Assume that there are not less than three blocks $J_{i_r}$ in $P^{-1}AP$ for every $r$, $r=1,2,\cdots, s$. Then 
\begin{equation*}
   n\ge 3(i_1+i_2+\cdots +i_s)=3t,
\end{equation*}
 this contradicts $3t-1\ge n\ge 2t$. Thus, there exists $k$, such that there are not more than two blocks $J_{i_k}$ in $P^{-1}AP$, where $1\le k\le s$. Hence, the linear span of $\mathnormal{E}, A, \cdots, A^{t-1}$ contains a matrix of rank $r\le 2$. Then, $l\left(\mathnormal{S}\right)\le 2n+n-2+t-1=3n-t-4\le\frac{7n}{2}-4$ by Theorem 3.4.
\end{proof}

\section{Acknowledgments}
I would like to express my sincere gratitude to Shitov for sharing their latest results with me via email. These findings have made significant contributions to this research field.

\end{document}